\documentclass[12pt]{amsart}
\usepackage{amsthm, amsmath, amssymb}
\usepackage{color}
\usepackage{graphicx}
\usepackage{url}
\usepackage{hyperref}

\newtheorem{theorem}{Theorem}[section]
\newtheorem{lemma}[theorem]{Lemma}

\newtheorem{corollary}[theorem]{Corollary}
\newtheorem*{observation}{Observation}

\theoremstyle{remark}
\newtheorem*{remark}{Remark}

\theoremstyle{definition}
\newtheorem*{definition}{Definition}

\newtheorem{example}{Example}[section]
\newtheorem{exercise}{Exercise}

\newtheorem{claim}[theorem]{Claim}

\newcommand{\BR}{{\mathbb R}}
\newcommand{\BC}{{\mathbb C}}
\newcommand{\BZ}{{\mathbb Z}}
\newcommand{\BH}{{\mathbb H}}
\newcommand{\BT}{{\mathbb T}}

\newcommand{\CC}{{\mathcal C}}
\newcommand{\CN}{{\mathcal N}}
\newcommand{\CH}{{\mathcal H}}

\newcommand{\BP}{{\mathbb P}}
\newcommand{\BE}{{\mathbb E}}

\newcommand{\Var}{\text{Var}}
\newcommand{\Cov}{\text{Cov}}

\newcommand{\bi}{{\mathbf i}}
\newcommand{\bj}{{\mathbf j}}

\newcommand{\Tr}{\text{Tr}\,}
\newcommand{\sgn}{\text{sgn}}

\newcommand{\Ai}{\text{Ai}}

\newcommand{\del}{\partial}
\newcommand{\grad}{\nabla}
\newcommand{\ov}{\overline}

\newcommand{\one}{{\mathbf 1}}

\renewcommand{\Re}{\text{Re}\,}
\renewcommand{\Im}{\text{Im}\,}

\providecommand{\abs}[1]{\left\lvert#1\right\rvert}
\providecommand{\norm}[1]{\lVert#1\rVert}

\title{Random matrix theory}
\author[Kargin - Yudovina]{Slava Kargin\\Statistical Laboratory, University of Cambridge\\{\tt V.Kargin@statslab.cam.ac.uk}\\
 Elena Yudovina\\Statistical Laboratory, University of Cambridge\\{\tt E.Yudovina@statslab.cam.ac.uk}}
\date{Michaelmas 2011}

\begin{document}
\maketitle

\tableofcontents

\section{Introduction}

Random matrix theory is usually taught as a sequence of several graduate courses; we have 16 lectures, so we will give a very brief introduction.

Some relevant books for the course:
\begin{itemize}
\item G. Anderson, A. Guionnet, O. Zeitouni. An introduction to random matrices. \cite{anderson_guionnet_zeitouni10}
\item A. Guionnet. Large random matrices: lectures on macroscopic asymptotics.
\item M. L. Mehta. Random matrices.
\end{itemize}

The study of random matrices originated in statistics, with the investigation of sample covariance matrices, and in nuclear physics, with Wigner's model of atomic nuclei by large random matrices.

A random matrix is a matrix with random entries. Let us see what sorts of questions we can ask about this object and what tools we can use to answer them.

Questions:\\
\vspace{0.2cm}
\parbox{4.5in}{Distribution of eigenvalues on the global scale. Typically, the histogram of eigenvalues looks something like this:} \qquad \parbox{2in}{\raisebox{-\height}{\begin{picture}(0,0)%
\includegraphics{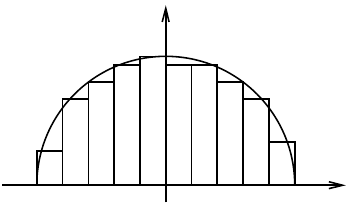}%
\end{picture}%
\setlength{\unitlength}{3947sp}%
\begingroup\makeatletter\ifx\SetFigFont\undefined%
\gdef\SetFigFont#1#2#3#4#5{%
  \reset@font\fontsize{#1}{#2pt}%
  \fontfamily{#3}\fontseries{#4}\fontshape{#5}%
  \selectfont}%
\fi\endgroup%
\begin{picture}(1674,972)(1264,-1621)
\end{picture}%
}}\\
What is the limiting shape of the histogram when the size of the matrix becomes large?

\vspace{0.2cm}
\parbox{4.5in}{Distribution of eigenvalues at the local scale. The histogram of spacings between the eigenvalues may look like this:} \qquad \parbox{2in}{\raisebox{-\height}{\begin{picture}(0,0)%
\includegraphics{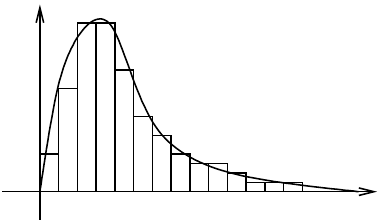}%
\end{picture}%
\setlength{\unitlength}{3947sp}%
\begingroup\makeatletter\ifx\SetFigFont\undefined%
\gdef\SetFigFont#1#2#3#4#5{%
  \reset@font\fontsize{#1}{#2pt}%
  \fontfamily{#3}\fontseries{#4}\fontshape{#5}%
  \selectfont}%
\fi\endgroup%
\begin{picture}(1824,1059)(1264,-1723)
\end{picture}%
}}\\
What is the limiting shape of this histogram?

\vspace{0.2cm}
Is there universality of these shapes with respect to small changes in the distribution of the matrix entries?\\
\vspace{0.2cm}
\parbox{2in}{Are the eigenvectors localized?} \qquad \parbox{4.5in}{\raisebox{-\height}{\begin{picture}(0,0)%
\includegraphics{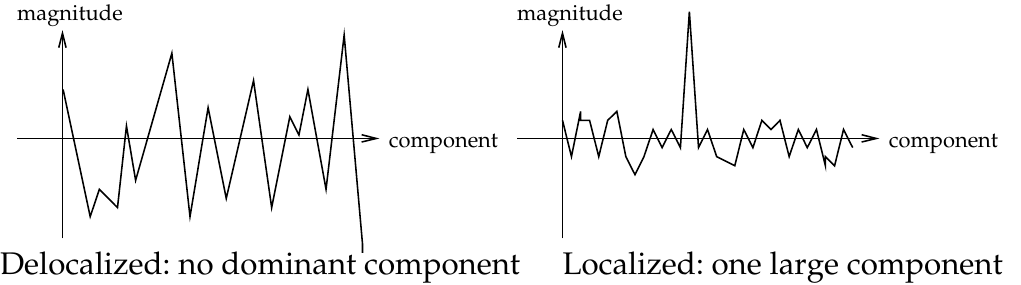}%
\end{picture}%
\setlength{\unitlength}{3947sp}%
\begingroup\makeatletter\ifx\SetFigFont\undefined%
\gdef\SetFigFont#1#2#3#4#5{%
  \reset@font\fontsize{#1}{#2pt}%
  \fontfamily{#3}\fontseries{#4}\fontshape{#5}%
  \selectfont}%
\fi\endgroup%
\begin{picture}(4884,1364)(4051,-1791)
\end{picture}%
}}\\

Approaches:
\begin{itemize}
\item Method of traces (combinatorial): for a function $f$, we have $\sum_{i=1}^N f(\lambda_i) = \Tr f(X)$, and the right-hand side can be studied using combinatorial methods.
\item Stieltjes transform method: we study the meromorphic function $\Tr\left(\frac{1}{X-z}\right)$, whose poles are the eigenvalues of $X$.
\item Orthogonal polynomials: in some cases, the distribution of eigenvalues can be written out explicitly and related to orthogonal polynomials, for which there is an extensive theory.
\item Stochastic differential equations: if the entries of the random matrix evolve according to some stochastic process described by an SDE, then the eigenvalues also evolve according to an SDE which can be analyzed.
\item Esoterica: Riemann-Hilbert, free probability, etc.
\end{itemize}

\newpage
\section{Method of Traces: Wigner random matrices}
Let $X$ be an $N \times N$ symmetric real-valued matrix. The matrix entries $X_{ij}$ are iid real-valued random variables (for $i \leq j$). Let $Z_{ij} = \sqrt{N} X_{ij}$. We assume $\BE Z_{ij} = 0$, $\BE Z_{ij}^2 = 1$, $\BE \abs{Z_{ij}}^k = r_k < \infty$ for all $k$ (and all $N$). (Most of these assumptions can be relaxed.)

\begin{definition}
The \emph{empirical measure of the eigenvalues} is $L_N = \frac1N \sum_{i=1}^N \delta_{\lambda_i}$, where $\lambda_i$ are eigenvalues of $X$.
\end{definition}

We will use the following notational convention: for a measure $\mu$ on $\BR$ and a function $f$, we write
\[
\langle \mu,f\rangle := \int_{\BR} f(x) \mu(dx).
\]
In particular, $\langle L_N, f \rangle = \frac1N \sum f(\lambda_i)$.

\begin{definition}
The \emph{density of states} $\ov{L_N}$ is the expected value of the random measure $L_N$.
\end{definition}
Note that the expectation is with respect to the randomness in $\{X_{ij}\}$.

Let $\sigma$ be the measure on $\BR$ with density $\frac{1}{2\pi} \sqrt{(4-x^2)_+}$ (so the support of $\sigma$ is confined to $[-2,2]$). The measure $\sigma$ is often called \emph{(Wigner's) semicircle law}. 
\begin{theorem}[Wigner]
For any bounded continuous function $f \in \CC_b(\BR)$,
\[
\BP(\abs{\langle L_N, f \rangle - \langle \sigma, f \rangle} > \epsilon) \to 0 ~~ \text{as $N \to \infty$}
\]
\end{theorem}
In other words, $L_N \to \sigma$ weakly, in probability.

\begin{proof}
The proof relies on two lemmas. The first asserts that the moments of $\ov{L_N}$ converges to the moments of $\sigma$; the second asserts that the moments of $L_N$ converges to the moments of $\ov{L_N}$ in probability.
\begin{lemma}\label{lemma 1}
$\ov{L_N} \to \sigma$ in moments, i.e. as $N \to \infty$, for any $k\geq 0$ we have
\[
\langle \ov{L_N}, x^k \rangle \to \langle \sigma, x^k \rangle = \begin{cases}
0, & k \text{ odd}\\
\frac{1}{\frac k2 + 1} \binom{k}{k/2}, & k \text{ even}.
\end{cases}
\]
\end{lemma}
(Note that this is a statement about convergence in $\BR$.)

The even moments of $\sigma$ coincide with the so-called Catalan numbers, which frequently arise in various enumeration problems in combinatorics.
\begin{lemma}\label{lemma 2}
As $N \to \infty$, for any $k$ we have
\[
\BP(\abs{\langle L_N, x^k \rangle - \langle \ov{L_N}, x^k \rangle} > \epsilon) \to 0
\]
\end{lemma}
(It will be sufficient to show that the variance of $\langle L_N, x^k \rangle$ tends to 0.)

Altogether, Lemmas~\ref{lemma 1} and \ref{lemma 2} show that moments of $L_N$ converge to moments of $sigma$ in probability. In a general situation, the convergence of moments does not imply the weak convergence of measures. However, this is true for a wide class of limiting measures, in particular for the measures that are uniquely determined by its moments. See, for example, Theorem 2.22 in Van der Vaart's ``Asymptotic Statistics'' \cite{van_der_vaart98}. 

For $\sigma$, we can give a direct proof. (See the outline of the general argument below.) Supposing Lemmas~\ref{lemma 1} and \ref{lemma 2} to hold, the proof proceeds as follows:

Let $f \in \CC_b(\BR)$ and $M:=\sup |f(x)|$. Observe that for every polynomial $Q$,
\begin{eqnarray*}
\BP(\abs{\langle L_N - \sigma, f\rangle}>\epsilon) & \leq &\BP(\abs{\langle L_N - \sigma, Q\rangle}>\epsilon/2) + \BP(\abs{\langle L_N - \sigma, f-Q\rangle}>\epsilon/2)\\
&\leq& \BP(\abs{\langle L_N - \ov{L_N}, Q\rangle}>\epsilon/4) + \BP(\abs{\langle \ov{L_N} - \sigma, Q\rangle}>\epsilon/4) \\
&& + \BP(\abs{\langle L_N - \sigma, f-Q\rangle}>\epsilon/2).
\end{eqnarray*}
 We claim that for every $\delta>0$ there exists a polynomial $Q$ and an integer $N_0$ so that the third term is less than $\delta$ for all $N\geq N_0$. Given this polynomial $Q$, the first two terms converge to 0 as $N \to \infty$ as a result of Lemmas~\ref{lemma 2} and \ref{lemma 1} respectively, and this proves the theorem. 

To prove that the claim holds, choose (by using the Weierstrass approximation theorem) $Q=\sum_{k=0}^K a_k x^k$ such that 
\begin{equation}
\sup_{\abs{x} \leq 5} \abs{f(x) - Q(x)} \leq \varepsilon/8.  
\end{equation}
Let $A=\max\{|a_k|\}$.

Note that since $\sigma$ supported on [-2,2], hence $|\langle \sigma,f-Q\rangle|\leq \varepsilon/8 <\varepsilon/4 $. Therefore, the claim will follow if for every $\delta>0$, we show that  there exists $N_0$, such that
\begin{equation}
\mathbb{P}\{|\langle L_N,f-Q\rangle|>\varepsilon/4 \}\leq\delta
\end{equation}
for all $N>N_0$. 

Note that
\begin{eqnarray*}
\mathbb{P}\{|\langle L_N,f-Q\rangle|>\varepsilon/4 \}&\leq&\mathbb{P}\{|\langle L_N,(f-Q)1_{|x|>5}\rangle|>\varepsilon/8 \} \\
&\leq& \mathbb{P}\left\{|\langle L_N,1_{|x|>5}\rangle|>\frac{\varepsilon}{8M(K+1)} \right\}\\
& & + \sum_{k=0}^K  \mathbb{P}\left\{|\langle L_N,|x|^k 1_{|x|>5}\rangle|>\frac{\varepsilon}{8A(K+1)} \right\}.
\end{eqnarray*}

 Let $C:=8(K+1)\max\{M,A\}$ and observe by the Markov inequality that
\begin{eqnarray*}
T_k:=\BP(\langle L_N, \abs{x}^k \one_{\{\abs{x} > 5\}} \rangle > \varepsilon/C) &\leq& \frac{C}{\varepsilon} \BE \langle L_N, \abs{x}^k \one_{\{\abs{x} > 5\}} \rangle \leq \frac{C\langle \ov{L_N}, x^{2k} \rangle}{\varepsilon 5^k}\\
&\leq&\frac{C}{\varepsilon} \left(\frac{4}{5}\right)^k. 
\end{eqnarray*}
for all integer $k\geq 0$ and $N\geq N_0(k)$. The last inequality follows by examining the form of the moments of $\sigma$ (to which the moments of $L_N$ converge).
 
 Note that $T_k$ is an increasing sequence. For every $\delta>0$ it is possible to find $L\geq K$ such that 
 \begin{equation*}
 T_L \leq \frac{\delta}{K+1}
 \end{equation*}
 for all $N\geq N_0=N_0(L)$.
 Then $T_k\leq \frac{\delta}{K+1}$ for all $k=0,\ldots, K$, which implies that 
 \begin{equation*}
\mathbb{P}\{|\langle L_N,f-Q\rangle|>\varepsilon/4 \}
\end{equation*}
This completes the proof. $\square$
\end{proof}   

Here is the outline of the high-level argument that shows that the convergence of moments implies the weak convergence of probability measures provided that the limiting measure is determined by its moments. Suppose for example that we want to show that $\ov{L_N}$ converges to $\sigma$ weakly.

Since the second moments of $\ov{L_N}$ are uniformly bounded, hence this sequence of measures is uniformly tight, by Markov's inequality. By Prohorov's theorem, each subsequence has a further subsequence that converges weakly to a limit $L$. It is possible to show that the moments of $L$ must equal to the limit of the moments of $\ov{L_N}$ (see Example 2.21 in van der Vaart). Since $\sigma$ is uniquely determined by its moments, hence $L$ must equal $\sigma$ for every choice of the subsequence. This is equivalent to the statement that  $\ov{L_N}$ converges weakly to $\sigma$.

We now move on to the proofs of the lemmas:
\begin{proof}[Proof of Lemma~\ref{lemma 1}]
Write
\[
\langle \ov{L_N}, x^k \rangle = \BE \left( \frac1N \sum_{i=1}^N \lambda_i^k \right) = \frac1N \BE \Tr(X^k) = \frac1N \sum_{\bi = (i_1, i_2, \dotsc, i_k)} \BE(X_{i_1, i_2} X_{i_2, i_3} \dotsc X_{i_k i_1}).
\]
Let $T_{\bi} = \BE(X_{i_1, i_2} X_{i_2, i_3} \dotsc X_{i_k i_1})$.

To each \emph{word} $\bi = (i_1, i_2, \dotsc, i_k)$ we associate a graph and a path on it. The vertices of the graph are the distinct indices in $\bi$; two vertices are connected by an edge if they are adjacent in $\bi$ (including $i_k$ and $i_1$); and the path simply follows the edges that occur in $\bi$ consecutively, finishing with the edge $(i_k, i_1)$. The important feature is that the expected value of $T_{\bi}$ depends only on the (shape of) the associated graph and path, and not on the individual indices.

\begin{observation}
If the path corresponding to $\bi$ traverses some edge only once, then $\BE(T_{\bi}) = 0$ (since we can take out the expectation of the corresponding term).
\end{observation}
Therefore, we are only interested in those graph-path pairs in which every edge is traversed at least twice. In particular, all such graphs will have at most $k/2$ edges, and (because the graph is connected) at most $k/2+1$ vertices.

\begin{exercise}
\label{exercise_moments}
If the graph associated to $\bi$ has $\leq k/2$ vertices, then $\BE(T_{\bi}) \leq c_k \BE \abs{X_{ij}}^k \leq \frac{B_k}{N^{k/2}}$ for some constants $c_k$, $B_k$. (Use H\"older's inequality.)
\end{exercise}
Now, for a fixed graph-path pair with $d$ vertices, the number of words giving rise to it (i.e. the number of ways to assign indices to the vertices) is $N(N-1)\dots(N-d+1) \approx N^d$. In addition, the number of graph-path pairs with $\leq k/2+1$ vertices and $\leq k/2$ edges is some finite constant depending on $k$ (but not, obviously, on $N$).

Consequently, by Exercise \ref{exercise_moments}  the total contribution of the words $\bi$ whose graphs have $\leq k/2$ vertices to $\frac1N \sum_{\bi} \BE(T_{\bi})$ is $O(N^{-1}N^{k/2}N^{-k/2})=O(N^{-1})$.

Therefore, we only care about graphs with exactly $k/2+1$ vertices, hence exactly $k/2$ edges, each of which is traversed twice. For each such graph, $\frac1N \BE(T_{\bi}) = \frac1N \BE(X_{ij}^2)^{k/2} = \frac{1}{N^{k/2+1}}$, while the number of ways to label the vertices of such a graph-path pair is $N(N-1)\dotsc(N-\frac k2 -1) \approx N^{k/2 + 1}$.
We conclude that the $k$th moment $\langle \ov{L_N},x^k \rangle$, for $k$ even, converges to the number of equivalence classes of pairs $G,P$ where $G$ is a tree with $k/2$ edges and $P$ is an oriented path along $G$ with $k$ edges traversing each edge of $G$ twice. (For $k$ odd it converges to 0.) 

Here the equivalence is with respect to a re-labelling of vertices. We now count the number of these objects.

An \textit{oriented tree} is a tree (i.e., a graph without closed paths of distinct edges) embedded in a
plane. A \textit{rooted tree} is a tree that has a special edge (``root''), which is oriented, that
is, this edge has a start and end vertices. Two oriented rooted tree are
isomorphic if there is a homeomorphism of the corresponding
planes which sends one tree to another one in such a way that the root goes to
the root and the orientations of the root and the plane are preserved.
Figure 1 shows an example of two trees that
are equivalent in the graph-theoretical sense as rooted trees, but not equivalent as oriented rooted trees.
(There is no isomorphism of planes that sends the first tree to the second
one.)

\begin{figure}[htbp]
\includegraphics[width=8cm]{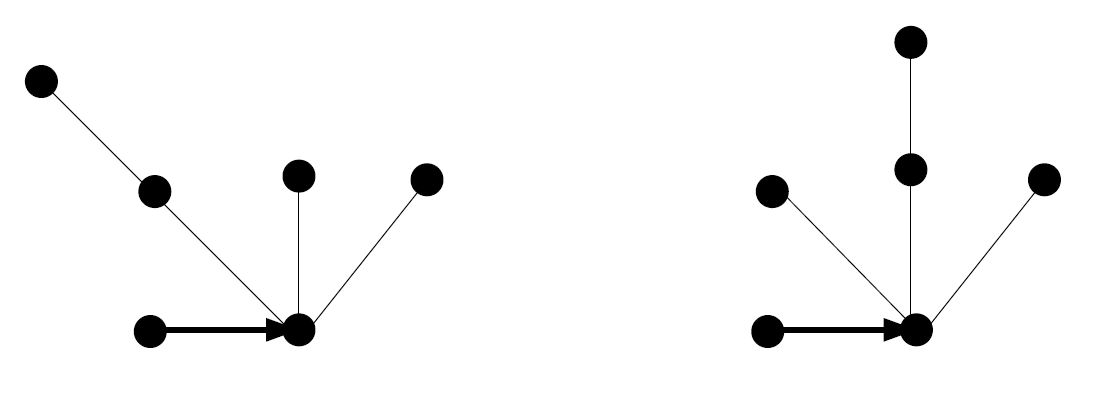} \label{figure_rooted_noniso}
\caption{Two non-isomorphic oriented rooted trees}
\end{figure}

Figure \ref{figure_rooted_trees} shows all non-isomorphic oriented rooted trees with 3 edges. 

\begin{figure}[tbph]
\includegraphics[width=10cm]{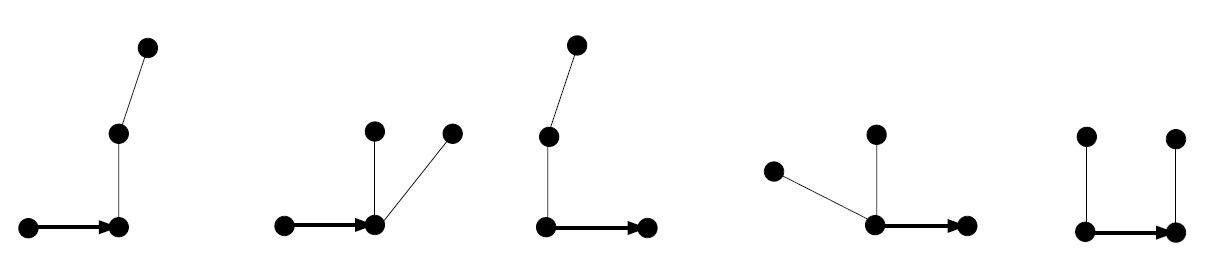}
\caption{Oriented rooted trees with 3 edges}
\label{figure_rooted_trees}
\end{figure}

\begin{claim}
There exists a bijection between equivalence classes of $(G,P)$ where $G$ is a tree with $k/2$ edges and $P$ is an oriented path along $G$ with $k$ edges traversing each edge of $G$ twice, and non-isomorphic oriented rooted trees.
\end{claim}
\begin{proof}
Given a path $P$ on $G$, we embed $G$ into the plane as follows: put the starting vertex at the origin, and draw unit-length edges out of it, clockwise from left to right, in the order in which they are traversed by $P$. Continue for each of the other vertices.

Conversely, given an embedding of $G$ with a marked directed edge, we use that edge as the starting point of the walk $P$ and continue the walk as follows: When leaving a vertex, pick the next available distance-increasing edge in the clockwise direction and traverse it. If there are no available distance-increasing edges, go back along the unique edge that decreases distance to the origin. It is easy to see that every edge will be traversed exactly twice.
\end{proof}

We can think of the embedded tree in terms of its thickening -- ``fat tree'' or ``ribbon graph''. Then this is actually a rule for the traversal of the boundary of the fat graph in the clockwise direction. 
Figure \ref{figure_fat_tree} shows an example of a fat tree. 

\begin{figure}[tbph]
\begin{center}
\includegraphics[width=4cm]{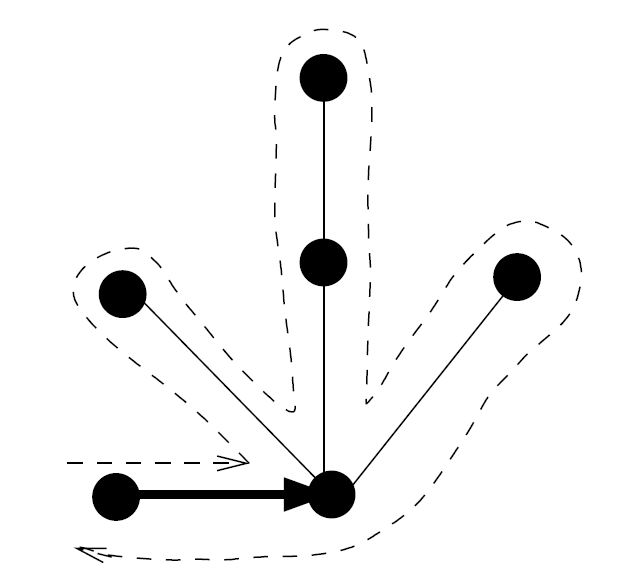}
\end{center}
\caption{A fat tree}
\label{figure_fat_tree}
\end{figure}

\begin{claim}
The oriented rooted trees are in bijection with \emph{Dick paths}, i.e. paths of a random walk $S_n$ on $\BZ$ such that $S_0 = S_k = 0$ and $S_j \geq 0$ for all $j = 0,1,\dotsc,k$.
\end{claim}
\begin{proof}
Let the embedding have unit-length edges. The random walk $S_n$ is simply measuring the distance (along the graph) of the $n$th vertex in the path from the root (the vertex at the source of the marked edge). If $S_n$ is increasing, we need to create a new edge to the right of all previously existing ones; if $S_n$ is decreasing, we need to follow the unique edge leading from the current vertex towards the root (``towards'' in the sense of graph distance).
\end{proof}

An example of this bijection is
shown in Figure \ref{figure_Dick_bijection}.

\begin{figure}[bpht]
\begin{center}
\includegraphics[width=7cm]{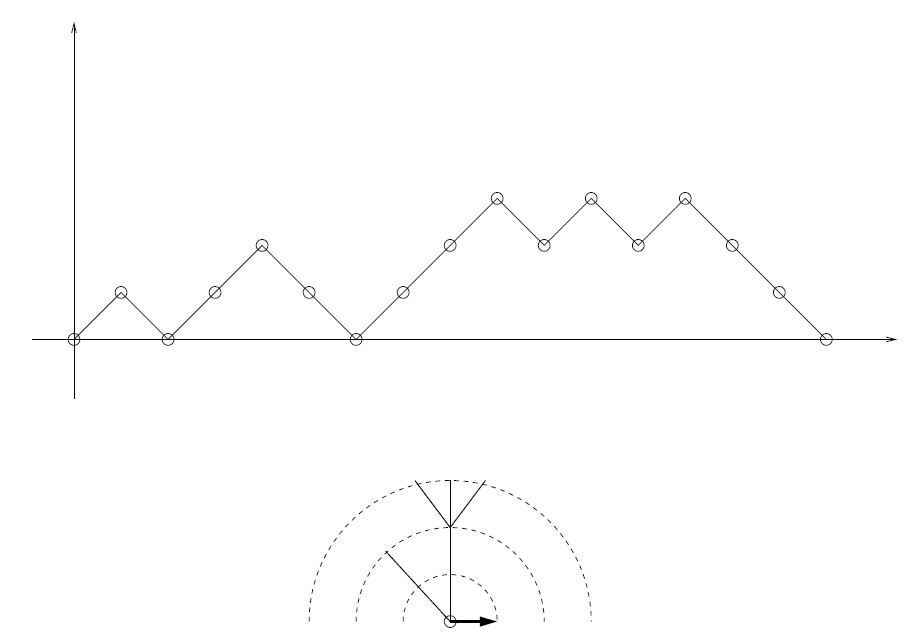}
\end{center}
\caption{A Dick path and the corresponding oriented rooted tree}
\label{figure_Dick_bijection}
\end{figure}

We now count the number of Dick paths as follows:
\begin{multline*}
\#\{\text{all paths with $S_0 = S_k = 0$}\} -\\
\#\{\text{paths with $S_0 = S_k = 0$ and $S_j = -1$ for some $j \in \{1,\dotsc,k-1\}$}\}
\end{multline*}
The first term is easily seen to be $\binom{k}{k/2}$ (we have $k$ jumps up and $k$ jumps down). For the second term, we argue that such paths are in bijection with paths starting at $0$ and ending at $-2$. Indeed, let $j$ be the last visit of the path to $-1$, and consider reflecting the portion of the path after $j$ about the line $y=-1$: it is easy to see that this gives a unique path terminating at $-2$ with the same set of visits to $-1$. The number of paths with $S_0 = 0$ and $S_k = -2$ is easily seen to be $\binom{k}{k/2-1}$ (we have $k-1$ jumps up and $k+1$ jumps down), so we conclude
\[
\lim_{N \to \infty} \langle {\ov L_N}, x^k \rangle = \begin{cases}
0, & k \text{ odd}\\
\binom{k}{k/2} - \binom{k}{k/2-1} = \frac{1}{k/2+1}\binom{k}{k/2}, & k \text{ even}
\end{cases} \qedhere
\]
\end{proof}

\begin{proof}[Proof of Lemma \ref{lemma 2}]
By Chebyshev's inequality, it suffices to show
\[
\Var(\langle L_N, x^k \rangle) \to 0 \text{ as $N \to \infty$}.
\]
We compute
\[
\Var(\langle L_N, x^k \rangle) = \frac{1}{N^2} \sum_{\bi,\bj} \BE(T_\bi T_\bj) - \BE(T_\bi) \BE(T_\bj) = \frac{1}{N^2} \sum_{\bi,\bj} \Cov(T_\bi,T_\bj).
\]
We associate a graph $G$ and a pair of paths $P_1 = P_1(\bi), P_2 = P_2(\bj)$ with the pair $\bi$, $\bj$: the vertices are the union of the indices in $\bi$ and $\bj$, the edges are the pairs $(i_s,i_{s+1})$ and $(j_t,j_{t+1})$ (with the convention $k+1=1$), the first path traverses the edges of $\bi$ in order, the second path traverses the edges of $\bj$ in order.

{\bf Observations:}
\begin{enumerate}
\item The graph $G$ may end up disconnected, but the corresponding covariance is 0.
\item In fact, for the covariance to be non-zero, every edge of $G$ mush be traversed at least twice by the union of $P_1$ and $P_2$. In particular, $G$ has at most $k$ edges, and at most $k+1$ vertices.
\item The number of labellings of the vertices of an equivalence class $(G,P_1,P_2)$ with at most $k+1$ vertices is at most $N^{k+1}$.
\item The number of equivalence classes of triples $(G,P_1,P_2)$ with at most $k+1$ vertices and $k$ edges is a finite constant (depending on $k$, but not on $N$).
\item Each $\Cov(T_\bi, T_\bj)$ is bounded by $O(N^{-k})$.
\end{enumerate}

We conclude that $\Var(\langle L_N, x^k \rangle) = O(N^{-1}) \to 0$ as required.
\end{proof}

\begin{exercise}
Show that in fact $\Var(\langle L_n, x^k \rangle) = O(N^{-2})$ by showing that the terms with $k+1$ vertices end up having covariance 0. (This is useful for showing almost sure convergence rather than convergence in probability, because the sequence of variances is summable.)
\end{exercise}

\subsection{Scope of the method of traces}

There are various ways in which the above analysis could be extended:
\begin{enumerate}
\item Technical extensions:
\begin{itemize}
\item We assumed that $Z_{ij} = \sqrt{N} X_{ij}$ has all moments. This is in fact unnecessary: we could truncate the entries $X_{ij}$ and use the fact that the eigenvalues of a matrix are Lipschitz with respect to the entries. All that is necessary is $\BE Z_{ij} = 0$, $\BE Z_{ij}^2 = 1$.
\item If $\BE Z_{ij}^2 = \sigma_{ij}$ depends on $i$ and $j$, then in general there is weak convergence $L_N \to \mu$ for some limiting measure $\mu$ (i.e., the variance of $L_N$ tends to 0), but $\mu$ need not be the semicircle law.
\end{itemize}
\item We can also find the laws for some other ensembles:
\begin{itemize}
\item
The \emph{Wishart ensemble} is $X^T X$ where $X$ is an $N \times M$ matrix with iid entries, $N \to \infty$ and $M/N \to \lambda$. (Moment assumptions as for the Wigner ensemble.) The limit law for the empirical distributions of its eigenvalues is the \emph{Marchenko-Pastur} law (see example sheet 1), given by
\[
f(x) = \frac{1}{2\pi x} \sqrt{(b-x)(x-a)}, \quad a = (1-\sqrt{\lambda})^2,~ b = (1+\sqrt{\lambda})^2
\]
plus an atom at 0 if $\lambda < 1$ (i.e. if $X^T X$ does not have full rank). This law looks like this:\\
\begin{center}
\begin{picture}(0,0)%
\includegraphics{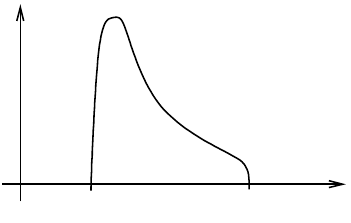}%
\end{picture}%
\setlength{\unitlength}{3947sp}%
\begingroup\makeatletter\ifx\SetFigFont\undefined%
\gdef\SetFigFont#1#2#3#4#5{%
  \reset@font\fontsize{#1}{#2pt}%
  \fontfamily{#3}\fontseries{#4}\fontshape{#5}%
  \selectfont}%
\fi\endgroup%
\begin{picture}(1674,1060)(1264,-1709)
\put(2429,-1669){\makebox(0,0)[lb]{\smash{{\SetFigFont{10}{12.0}{\rmdefault}{\mddefault}{\updefault}{\color[rgb]{0,0,0}$b$}%
}}}}
\put(1679,-1669){\makebox(0,0)[lb]{\smash{{\SetFigFont{10}{12.0}{\rmdefault}{\mddefault}{\updefault}{\color[rgb]{0,0,0}$a$}%
}}}}
\end{picture}%

\end{center}

This is more ``useful'' than Wigner analysis, because $X^T X$ can be thought of as sample covariances. For example, we can check whether we got something that looks like noise (in this case eigenvalues of the covariance matrix ought to approximately follow the Marchenko-Pastur law) or has eigenvalues which are outliers. For more on this, see  N. El Karoui, \cite{el_karoui08}, ``Spectrum estimation for large dimensional covariance matrices using random matrix theory'', \emph{Annals of Statistics} 36:6 (2008), pp. 2757-2790.
\end{itemize}
\item Behaviour of the largest eigenvalue: consider $\frac1N \BE \Tr(X_N^{k_N})$, where the moment $k_N$ depends on $N$. If $k_N \to \infty$ as $N \to \infty$, this is dominated by $\frac1N \lambda_{\max}^{k_N}$. The difficulty with the analysis comes from the fact that if $k_N \to \infty$ quickly, then more graphs have nonnegligible contributions, and the combinatorics becomes rather nasty. E.g., F\"{u}redi-Koml\'{o}s \cite{furedi_komlos81} showed that (in particular) $\BP(\lambda_{\max} > 2+\delta) \to 0$ as $N \to \infty$. (In fact, they showed this with some negative power of $N$ in the place of $\delta$.) Soshnikov (in \cite{sinai_soshnikov98} and \cite{soshnikov99})  has extended this analysis to determine the asymptotic distribution of the largest eigenvalue of the Wigner ensemble.
\end{enumerate}

However, the trace method has various limitations, in particular:
\begin{enumerate}
\item Difficult to say anything about the local behaviour of the eigenvalues (e.g., their spacings) away from the edge of the distribution (``in the bulk''), because eigenvalues in the bulk aren't separated by moments;
\item Difficult to get the speed of convergence, and
\item Says nothing about the distribution of eigenvectors.
\end{enumerate}

\newpage

\section{Stieltjes transform method}
\subsection{Stieltjes transform}
\begin{definition}
Suppose $\mu$ is a nonnegative finite measure on $\BR$. The Stieltjes transform of $\mu$ is
\[
g_\mu(z) = \int_\BR \frac{1}{x-z} \mu(dx), \quad z \in \BC \setminus \BR.
\]
\end{definition}

If all the moments of $\mu$ are finite (and the resulting series converges), we can rewrite this as
\[
g_\mu(z) = \int_\BR -\frac1z \frac1{1-x/z} \mu(dx) = -\frac1z \sum_{k=0}^\infty \frac{m_k}{z^k},
\]
where $m_k = \int x^k \mu(dx)$ is the $k$th moment of $\mu$.

If $\mu = L_N$, then
\[
g_{L_N}(z) = \frac1N \sum_{i=1}^N \frac1{\lambda_i-z} = \frac1N \Tr \frac{1}{X_N - zI},
\]
where $(X_N-z)^{-1}$ is the \emph{resolvent} of $X_N$, for which many relations are known.

The important thing about the Stieltjes transform is that it can be inverted.
\begin{definition} A measure $\mu$ on $\BR$ is called a \textit{sub-probability measure} if $\mu(\BR) \leq 1$.
\end{definition}  
\begin{theorem}
Let $\mu$ be a sub-probability measure. Consider an interval $[a,b]$ such that $\mu\{a\} =\mu\{b\} = 0$. Then
\[
\mu[a,b] = \lim_{\eta \to 0} \int_a^b \frac1\pi \Im g_\mu(x+i \eta) dx
\]
\end{theorem}

\begin{proof}
First, we massage the expression:
\begin{multline*}
\int_a^b \frac1\pi \Im g_\mu(x+i \eta) dx = \int_a^b \frac1\pi \int_{-\infty}^\infty \frac{\eta}{(\lambda-x)^2 + \eta^2} \mu(d\lambda) dx =\\
\int_{-\infty}^\infty \frac1\pi \left( \tan^{-1}\left(\frac{b-\lambda}{\eta}\right) - \tan^{-1}\left(\frac{a-\lambda}{\eta}\right) \right) \mu(d\lambda).
\end{multline*}
(Note: $\tan^{-1}$ is the inverse tangent function.)

Let
\[
R(\lambda) = \frac{1}{\pi}\left[\tan^{-1}\left(\frac{b-\lambda}{\eta}\right) - \tan^{-1}\left(\frac{a-\lambda}{\eta}\right)\right];
\]
a plot of this function looks as follows:\\
\begin{center}
\begin{picture}(0,0)%
\includegraphics{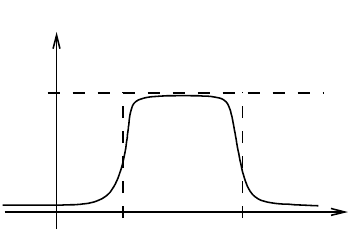}%
\end{picture}%
\setlength{\unitlength}{3947sp}%
\begingroup\makeatletter\ifx\SetFigFont\undefined%
\gdef\SetFigFont#1#2#3#4#5{%
  \reset@font\fontsize{#1}{#2pt}%
  \fontfamily{#3}\fontseries{#4}\fontshape{#5}%
  \selectfont}%
\fi\endgroup%
\begin{picture}(1687,1196)(1255,-1709)
\put(1403,-1030){\makebox(0,0)[lb]{\smash{{\SetFigFont{10}{12.0}{\rmdefault}{\mddefault}{\updefault}{\color[rgb]{0,0,0}1}%
}}}}
\put(2942,-1583){\makebox(0,0)[lb]{\smash{{\SetFigFont{10}{12.0}{\rmdefault}{\mddefault}{\updefault}{\color[rgb]{0,0,0}$\lambda$}%
}}}}
\put(1333,-633){\makebox(0,0)[lb]{\smash{{\SetFigFont{10}{12.0}{\rmdefault}{\mddefault}{\updefault}{\color[rgb]{0,0,0}$R(\lambda)$}%
}}}}
\put(1827,-1669){\makebox(0,0)[lb]{\smash{{\SetFigFont{10}{12.0}{\rmdefault}{\mddefault}{\updefault}{\color[rgb]{0,0,0}$a$}%
}}}}
\put(2395,-1669){\makebox(0,0)[lb]{\smash{{\SetFigFont{10}{12.0}{\rmdefault}{\mddefault}{\updefault}{\color[rgb]{0,0,0}$b$}%
}}}}
\end{picture}%

\end{center}

Some facts:
\begin{enumerate}
\item $0 \leq R(\lambda) \leq 1$
\item $\tan^{-1}(x) - \tan^{-1}(y) = \tan^{-1}\left(\frac{x-y}{1+xy}\right)$
and therefore
\[
R(\lambda) = \frac1\pi \tan^{-1} \left(\frac{\eta(b-a)}{\eta^2+(b-\lambda)(a-\lambda)}\right).
\]
\item $\tan^{-1}(x) \leq x$ for $x \geq 0$.
\end{enumerate}

Let $\delta$ be such that $\mu[a-\delta,a+\delta] \leq \epsilon/5$ and $\mu[b-\delta,b+\delta] \leq \epsilon/5$. Now,
\begin{multline*}
\abs{\int_\BR \one_{[a,b]}(\lambda) \mu(d\lambda) - \int_\BR R(\lambda) \mu(d\lambda)} \leq\\
\frac{2\epsilon}{5} + \int_{-\infty}^{a-\delta} R(\lambda) \mu(d\lambda) + \int_{b+\delta}^\infty R(\lambda) \mu(d\lambda) + \int_{a+\delta}^{b-\delta} (1-R(\lambda)) \mu(d\lambda).
\end{multline*}
For the first term, if $\lambda < a,b$ then the argument of the arctangent in $R(\lambda)$ is positive, so we apply the third fact to get
\[
\int_{-\infty}^{a-\delta} R(\lambda) \mu(d\lambda) \leq \frac1\pi \eta \int_{-\infty}^{a-\delta} \frac{b-a}{\eta^2 + (b-\lambda)(a-\lambda)} \mu(d\lambda).
\]
(Similarly for the second term, where $\lambda > a,b$.) It's not hard to check that the integral is finite, and bounded uniformly in all small $\eta$. Hence the first two terms are $\leq \frac\epsilon5$ for sufficiently small $\eta$.

Finally, for the third term, we have
\begin{multline*}
\int_{a+\delta}^{b-\delta}\left(1-\frac1\pi \tan^{-1}\left(\frac{b-\lambda}{\eta}\right) + \frac1\pi \tan^{-1}\left(\frac{a-\lambda}{\eta}\right)\right) \mu(d\lambda) \leq\\
\int_{a+\delta}^{b-\delta} \left(1-\frac2\pi \tan^{-1} \left(\frac\delta\eta\right) \right) \mu(d\lambda)
\end{multline*}
Now, as $\eta \to 0+$ we have $\tan^{-1}(\delta/\eta) \to \frac\pi2$ (from below), and in particular the integral will be $\leq \frac\epsilon5$ for all sufficiently small $\eta$.

Adding together the bounds gives $\leq \epsilon$ for all sufficiently small $\eta$, as required.
\end{proof}

A somewhat shorter version of the proof can be obtained by observing that as $\eta \rightarrow 0$, $R(\lambda)\rightarrow 1$ for $\lambda \in (a,b)$, $R(\lambda)\rightarrow 0$ for $\lambda \in [a,b]^c$, and $R(\lambda)\rightarrow 1/2$ for $\lambda \in \{a,b\}$. In addition, $R(\lambda)$ can be majorized uniformly in $\eta$ by a positive integrable function. This is because $R(\lambda)$ is uniformly bounded and its tails are $\sim \eta/\lambda^2$. Hence, we can apply the dominated convergence theorem (e.g., property (viii) on p. 44 in \cite{chung01}) and obtain 
\begin{equation}
\frac{1}{\pi}\int_{\BR}R(\lambda)\mu(d\lambda)\rightarrow \mu[a,b].
\end{equation} 

\begin{corollary}
If two subprobability measures $\mu$, $\nu$ have $g_\mu=g_\nu$, then $\mu=\nu$.
\end{corollary}

\begin{theorem}
\label{theorem_convergence_Stieltjes_and_measure}
Let $\{\mu_n\}$ be a sequence of probability measures, with Stieltjes transforms $g_{\mu_n}(z)$. Suppose there is a probability measure $\mu$ with Stieltjes transform $g_\mu(z)$, such that $g_{\mu_n}(z) \to g_\mu(z)$ holds at every $z \in A$, where $A \subseteq \BC$ is a set with at least one accumulation point. Then $\mu_n \to \mu$ weakly.
\end{theorem}
It's reasonably clear that if $\mu_n$ converges weakly, then the limit must be $\mu$; the main issue is showing that it converges.
\begin{proof}
\begin{definition}
We say that a sequence of sub-probability measures $\mu_n$ converges to a sub-probability measure $\mu$ \emph{vaguely} if for all $f \in \CC_0(\BR)$ we have
\[
\int_\BR f \mu_n(dx) \to \int_\BR f \mu(dx)
\]
Here, the space $\CC_0$ consists of continuous functions $f$ with $f(x) \to 0$ as $\abs{x} \to \infty$.
\end{definition}
Recall that weak convergence is the same statement for \emph{all} continuous bounded functions $f$.

Fact: a sequence of sub-probability measures $\{\mu_n\}$ always has a limit point with respect to vague convergence (Theorem 4.3.3 on p. 88 in \cite{chung01}).

\begin{example}
The sequence of measures in Figure¬\ref{fig:vague} has no limit points with respect to weak convergence, but converges to 0 with respect to vague convergence.
\begin{figure}[htbp]
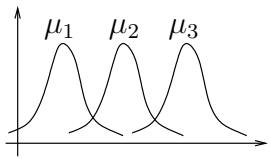
\caption{$\{\mu_n\}$ has no limit points with respect to weak convergence, but $\mu_n \to 0$ vaguely.}
\label{fig:vague}
\end{figure}
\end{example}

Fact 2: if a sequence of probability measures $\mu_n$ converges to a probability measure $\mu$ vaguely, then $\mu_n \to \mu$ weakly (Theorem 4.4.2 on p. 93 in \cite{chung01}).

We can now finish the proof. We will show that all vague limit points of $\{\mu_n\}$ have the same Stieltjes transform, $g_\mu$. Since the Stieltjes transform is invertible, all vague limit points of $\{\mu_n\}$ are $\mu$, and hence $\mu_n \to \mu$ weakly.

Thus, it remains to show that if $\nu$ is a vague limit point of $\{\mu_n\}$, then $g_\nu(z) = g_\mu(z)$.

First, observe that since the Stieltjes transform is holomorphic on the upper half-plane $\BH = \{z \in \BC: \Im z > 0\}$, the convergence $g_{\mu_n}(z) \to g_\mu(z)$ on $A$ implies $g_{\mu_n}(z) \to g_\mu(z)$ on the entire upper half-plane $\BH$. (In particular, we might as well have started with convergence on $\BH$ as the theorem assumption.)

Now, since $\Im\left(\frac{1}{x-z}\right)$ and $\Re\left(\frac{1}{x-z}\right)$ belong to $\CC_0(\BR)$ (for each $z$ in the upper half-plane), we must have
\[
\int \Im\left(\frac{1}{x-z}\right) \mu_n(dx) \to \int \Im\left(\frac{1}{x-z}\right) \nu(dx)
\]
by definition of vague convergence, and similarly for the real part. We conclude $g_{\mu_n}(z) \to g_\nu(z)$ for every $z \in \BH$, which finishes the proof.
\end{proof}

\subsection{Application to the analysis of the Gaussian Wigner ensemble}
 Let $X_N$ be a Gaussian Wigner matrix, that is, a symmetric matrix with i.i.d entries Gaussian $X_{ij}$, such that $\BE X_{ij}=0$ and $\BE X_{ij}^2=1$. Our goal is to show that the eigenvalue distribution of $X_N$ weakly converges in probability to the Wigner semicircle law.
 First, let us collect some useful facts about operator resolvents.
\begin{definition}
The \emph{resolvent} of $X$ is $G_X(z) = (X-z)^{-1}$.
\end{definition}
The Stieltjes transform of $L_N$ is
\[
g_{L_N}(z) = \frac1N \sum_{i=1}^N \frac{1}{\lambda_i - z} = \frac1N \Tr G_{X_N}(z)
\]
and in particular,
\[
\BE g_{L_N}(z) = g_{\ov{L_N}}(z) = \frac1N \BE \Tr(G_{X_N}(z))
\]

\begin{lemma}[Resolvent identity]
\[
G_{X+A}(z) - G_X(z) = -G_{X+A}(z)AG_x(z)
\]
\end{lemma}
\begin{proof}
Multiply by $X-z$ on the right, and by $X+A-z$ on the left.
\end{proof}
\begin{corollary}
Taking $A=-X$, we get
\[
-\frac1z - G_X(z) = \frac1z X G_X(z)
\]
so
\[
G_X(z) = -\frac1z + \frac1z X G_X(z)
\]
\end{corollary}
Note: Here $1/z$ is a shortcut for $(1/z)I_N$, where $I_N$ is the $N\times N$ identity matrix.
\begin{corollary}\label{coroll 2}
\[
\frac{\del G_{uv}}{\del X_{ij}} = -G_{ui}G_{jv} - G_{uj} G_{iv}
\]
\end{corollary}
(The question of what $A$ you need to get this is left as an exercise to the reader, but at a guess $A$ should have 0 everywhere except the $ij$th -- and possibly $ji$th -- entry, in which it should have an $\epsilon$.)

\begin{lemma}\label{lemma:E(x*f(x))}
If $\xi \sim \CN(0,\sigma^2)$ and $f$ is a differentiable function which grows no faster than a polynomial, then $\BE \xi f(\xi) = \sigma^2 \BE(f'(\xi))$.
\end{lemma}

We now continue the analysis of the Wigner ensemble:
\begin{align*}
\frac1N \BE \Tr(G_{X_N}(z)) &= -\frac1z + \frac{1}{Nz}\BE \Tr(X G_X)\\
&= -\frac1z +\frac{1}{Nz} \sum_{i,j} \BE(X_{ij} G_{ji})\\
&= -\frac1z + \frac{1}{N^2 z} \sum_{i,j} \BE \left(\frac{\del G_{ji}}{\del X_{ij}} \right)\\
&= -\frac1z + \frac{1}{N^2 z} \sum_{i,j} \BE(-G_{ji}G_{ij}-G_{jj}G_{ii})\\
&= -\frac1z - \frac{1}{N^2 z} \BE(\Tr(G^2)) - \frac1z \BE((\frac1N \Tr G)^2).
\end{align*}
Here, the third line follows from Lemma~\ref{lemma:E(x*f(x))}, and the fourth line follows from Corollary~\ref{coroll 2}.

Note that $\frac{1}{N^2 z} \Tr(G^2) = \frac{1}{Nz} \frac1N \sum_{i=1}^N \frac{1}{(\lambda_i - z)^2}$, and every term in the sum is bounded by $\frac{1}{\eta^2}$, where $\eta = \Im z$. In particular, as $N \to \infty$, this term $\to 0$.

We conclude
\[
\BE g(z) = -\frac1z - \frac1z \BE(g(z)^2) + E_N,
\]
where the error term $E_N$ satisfies $E_N \to 0$ as $N \to \infty$.

Now, if we had $\Var g(z) \to 0$ as $N \to \infty$, we could write
\[
\BE g(z) = -\frac1z - \frac1z \BE(g(z))^2 + \tilde E_N,
\]
for some other error term $\tilde E_N \to 0$ as $N \to \infty$. After some technical work, it can be shown that the solution of this equation converges to the solution of
\[
s(z) = -\frac1z - \frac1z s(z)^2
\]
for every $z \in \BH$, as $N \to \infty$. Since this is the Stieltjes transform of the semi-circle law, by using Theorem \ref{theorem_convergence_Stieltjes_and_measure} we conclude that $\BE L_N$ converges weakly to the semicircle law.

In addition, if $\mathrm{Var}g(z)\to 0$ as $N \to \infty$, then $g(z)-\BE g(z) \to 0$ in probability, which allows us to conclude that $L_N-\ov{L}_N \to 0$ weakly in probability.  

 Hence, all that remains to show is that
\[
\Var(g(z))=\Var\left(\frac1N\sum_{i=1}^N \frac{1}{\lambda_i - z} \right) \to 0.
\]
We have here a sum of many random terms. The problem is that the $\lambda_i$ are not at all independent. One way to handle this problem is to use concentration inequalities.

\subsection{Concentration inequalities; LSI}

\begin{definition}
Let $\mu$ be a measure on $\BR^m$. Define $W^{1,2}(\mu)$ as the space of differentiable functions $f$ such that $f \in L^2(\mu)$ and $\norm{\grad f}_2 \in L^2(\mu)$.
\end{definition}

\begin{definition}
A probability measure $\mu$ on $\BR^m$ is called \emph{LSI with constant $c$} if for every $f \in W^{1,2}(\mu)$ we have
\[
\int_{\BR^m} f^2 \log\left(\frac{f^2}{\int f^2 d\mu}\right) d\mu \leq 2c \int_{\BR^m} \norm{\grad f}_2^2 d\mu.
\]
\end{definition}
The name ``LSI'' here stands for ``logarithmic Sobolev inequality''.

We compare this with the Poincar\'e inequality:
\[
\Var f \leq \tilde c \int \norm{\grad f}_2^2 d\mu.
\]

We have the following facts:
\begin{enumerate}
\item Gaussian measure on $\BR$ is LSI.
\item\label{fact 4} If the law of random variable $X$ is LSI with constant $c$, then the law of $\alpha X$ is LSI with constant $\alpha^2 c$.
\item\label{fact 5} If $\mu$ is LSI with constant $c$, then the product measure $\mu \otimes \mu \otimes \dotsc \otimes \mu$ is LSI with the same constant $c$.
\item There is a Bobkov-G\"{o}tze criterion for distributions with a density to satisfy LSI with some finite constant $c$ (see \cite{bobkov_gotze99}). One useful case is when the density has the form $const\times \exp(-g(x))$, with twice differentiable $g(x)$ and $c_1\leq g''(x) \leq c_2$ for some positive constants $c_1$, $c_2$. 
\item A discrete measure is not LSI under this definition. However, there is also a discrete version of LSI, which is satisfied e.g. by the Bernoulli distribution. See example sheet 2.
\item There is a large literature about logarithmic Sobolev inequalities which are useful in the study of Markov processes. See review by Gross (\cite{gross93}) or lecture notes by Guionnet and Zegarlinski (http://mathaa.epfl.ch/prst/mourrat/ihpin.pdf) The topic is important and recently found applications in random matrix theory through the study of eigenvalues of matrices whose entries follow the Ornstein-Uhlebeck process. However, we are not concerned with these issues in our lectures. 
\end{enumerate}

For us, $X_N$ has a law which is a product of Gaussian measures with variance $1/N$, and hence by Facts~\ref{fact 4} and \ref{fact 5}, this law is LSI with constant $c/N$, where $c$ is the LSI constant of the standard Gaussian measure. The main idea is that the probability of a large deviation of a Lipschitz function $F(x_1,\ldots,x_n)$ from its mean can be efficiently estimated if the probability distribution of $x_1,\ldots,x_n$ is LSI.

\begin{definition}
A function $f: \BR^m \to \BR$ is \textit{Lipschitz} with constant $L$ if 
\begin{equation}
\sup_{x \neq y} \frac{\abs{f(x) - f(y)}}{\norm{x-y}} \leq L.
\end{equation} 
\end{definition} 

\begin{lemma}[Herbst]\label{herbst lemma}
Suppose $\mu$ is a measure on $\BR^m$ which is LSI with constant $c$, and $F: \BR^m \to \BR$ is Lipschitz with constant $L$. Then
\begin{enumerate}
\item $\BE[e^{\lambda (F - \BE F)}] \leq e^{c\lambda^2 L^2/2}$, for all $\lambda$
\item $\BP(\abs{F - \BE F} > \delta) \leq 2\exp(-\frac{\delta^2}{2cL^2})$
\item $\mathrm{Var} F \leq 4cL^2$.
\end{enumerate}
\end{lemma}
We will prove this lemma a bit later. In order to apply this lemma we need to find out what is the Lipschitz constant of the Stieltjes transform.

\begin{lemma}[Hoffman-Wielandt inequality]\label{hw ineq}
If $A$ and $B$ are symmetric $N \times N$ matrices with eigenvalues
\[
\lambda^A_1 \leq \lambda^A_2 \leq \dotsc \leq \lambda^A_N, \quad \lambda^B_1 \leq \lambda^B_2 \leq \dotsc \leq \lambda^B_N
\]
then
\[
\sum_{i=1}^N (\lambda^A_i - \lambda^B_i)^2 \leq \mathrm{Tr}(A-B)^2 \leq 2\sum_{1 \leq i \leq j \leq N} (A_{ij} - B_{ij})^2.
\]
\end{lemma}
The proof can be found in \cite{anderson_guionnet_zeitouni10}.

\begin{corollary}
The eigenvalues of a matrix $X$ are Lipschitz functions of the matrix entries $X_{ij}$, with constant $\sqrt{2}$.
\end{corollary}

\begin{corollary}
$g(X) = \frac1N \sum_{i=1}^N \frac{1}{\lambda_i - z}$ is Lipschitz (as a function of $X$!) with constant $c/\sqrt{N}$, where $c>0$ depends only on $\Im z$.
\end{corollary}
(See problem 7 in example sheet 2.)

We now use these results to finish the proof of the Wigner law for Gaussian Wigner matrices: recall we still needed to show that $\mathrm{Var} g_N(z) \to 0$. We know that the joint measure of the entries of $X_N$ is LSI with constant $c/N$, and $g_N(z)$ is Lipschitz (in $X$) with constant $L = L(\eta)/sqrt(N)$, where $\eta = \Im z$. By applying the Herbst lemma, $\mathrm{Var} g_N(z) \leq \frac{2cL^2(\eta)}{N^2} \to 0$ for all $z \in \BH$ as $N \to \infty$.

\begin{remark}
Note that this means that $\Var \sum_{i=1}^N \frac{1}{\lambda_i - z} \to C$ as $N \to \infty$. (This is surprising, since usually we would normalize a sum of $N$ random terms by $1/\sqrt{N}$. No normalization is needed at all in this case!)
\end{remark}

\begin{proof}[Proof of Herbst Lemma~\ref{herbst lemma}]
Let $A(\lambda) = \log \BE \exp(2\lambda(F - \BE F))$; we want to show $A(\lambda) \leq 2c\lambda^2 L$.

Applying LSI to $f = \exp(\lambda(F - \BE F))$, we have
\[
\int e^{2\lambda(F - \BE F)} \log \frac{e^{2\lambda(F-\BE F)}}{e^{A(\lambda)}} d\mu \leq 2c \int \lambda^2 e^{2\lambda(F-\BE F)}\norm{\grad F}_2^2 d\mu
\]

On the left-hand side, we have
\[
\int 2\lambda(F-\BE F) e^{2\lambda(F - \BE F)} d\mu - A(\lambda) \int e^{2\lambda(F - \BE F)} d\mu
\]
Note that the first term is $\lambda \frac{\del}{\del \lambda}(e^{2\lambda(F-\BE F)})$; interchanging integration and differentiation, we conclude that the left-hand side is
\[
\lambda (e^{A(\lambda)})' - A(\lambda) e^{A(\lambda)} = e^{A(\lambda)}\left(\lambda A'(\lambda) - A(\lambda)\right) = e^{A(\lambda)}\lambda^2\left(\frac{A(\lambda)}{\lambda}\right)'
\]
(where $'$ designates differentiation with respect to $\lambda$).

The right-hand side is bounded by $2c \lambda^2 e^{A(\lambda)} L^2$, since $\norm{\grad f}_2^2 \leq L^2$.

Consequently,
\[
\left(\frac{A(\lambda)}{\lambda}\right)' \leq 2cL^2.
\]
It's not hard to check by Taylor-expanding that $A(\lambda) = O(\lambda^2)$, so $A(\lambda)/\lambda \to 0$ as $\lambda \to 0$. Consequently,
\[
A(\lambda) \leq 2cL^2 \lambda^2
\]
as required. Two other statements easily follow from the first one (see exercise sheet 2). 
\end{proof}

\subsection{Wigner's theorem for non-Gaussian Wigner matrices by the Stieltjes transform method}

Recall the definition of the Wigner matrix. Matrix $X_N$ be an $N \times N$ symmetric real-valued matrix with matrix entries $X_{ij}$. We assume that $Z_{ij} = \sqrt{N} X_{ij}$ are iid real-valued random variables (for $i \leq j$) which are taken from a probability distribution, the same one for every $N$.

 We are going to prove the following result by the Stieltjes transform method. 

\begin{theorem}[Wigner]
Let $X_N$ be a sequence of Wigner matrices, such that $\BE Z_{ij} = 0$, $\BE Z_{ij}^2 = 1$ and the measure of $Z_{ij}$ is LSI, and let $L_N$ be its empirical measure of eigenvalues. Then, for any bounded continuous function $f \in \CC_b(\BR)$,
\[
\BP(\abs{\langle L_N, f \rangle - \langle \sigma, f \rangle} > \epsilon) \to 0 ~~ \text{as $N \to \infty$}
\]
\end{theorem}

\begin{remark} The assumption that the measure of $Z_{ij}$ is LSI can be omitted and the Theorem still remains valid. The proof of this is somewhat more involved. See Remark at the end of the proof and example sheet 2.
\end{remark}

Proof: The general idea of the proof is the same as for Gaussian Wigner matrices. Two things that need to be proved is that
\[
\BE g_N(z) = -\frac1z - \frac1z \BE(g_N(z))^2 + \tilde E_N,
\]
for some error term $\tilde E_N \to 0$ as $N \to \infty$, and that $\mathrm{Var} g_N(z) \to 0.$ The second claim follows immediately from the assumption that the measure of $Z_{ij}$ is LSI. (Our proof of this result in the previous section does not depend on the assumption that the entries are Gaussian.) 

The proof of the first claim uses some matrix identities.

\begin{lemma}
Let $X$ be a square block matrix $\begin{pmatrix} A & B\\C & D \end{pmatrix}$, where $A$ is square and invertible. Then
\[
\det X = \det A \det(D - CA^{-1} B).
\]
\end{lemma}
\begin{proof}
Carry out the (block) UL decomposition:
\[
\begin{pmatrix} A & B\\ C & D \end{pmatrix} = \begin{pmatrix}A & 0\\ C & D-CA^{-1}B \end{pmatrix}\begin{pmatrix} I & A^{-1}B\\0 & I \end{pmatrix}.
\]
\end{proof}

Let $x_i$ denote the $i$th column of $X$ without the $i$th entry. Let $X^{(i)}$ denote the matrix $X$ without the $i$th row and column.

\begin{lemma}
\[
\left((X-z)^{-1}\right)_{ii} = \left(X_{ii} - z - x_i^T (X^{(i)} - z) x_i \right)^{-1}.
\]
\end{lemma}
\begin{proof}
Apply Cramer's rule:
\[
((X-z)^{-1})_{ii} = \frac{\det(X^{(i)} - z)}{\det(X-z)}
\]
Now use Lemma 1 with $A = X^{(i)} - z$, $D = X_{ii}-z$ (note that the lemma is only directly applicable if $i=n$, but that's a trivial reordering of the basis).
\end{proof}
Hence,
\[
g_N(z) = \frac1N \sum_i ((X-z)^{-1})_{ii} = \frac1N \sum_i \frac{1}{X_{ii} - z - x_i^T (X^{(i)} - z)^{-1} x_i}.
\]
We want to get rid of the $X_{ii}$ in the above expression; this is done by the following
\begin{lemma}
Let $X_N$ by a sequence of symmetric matrices with independent entries s.t. $\BE X_{ij} = 0$ and $\BE (\sqrt{N} X_{ij}) = 1$. Let $\tilde X_N$ be defined by setting $\tilde X_{ij} = 0$ if $i = j$ and $X_{ij}$ otherwise. Then $\abs{g_{\tilde X_N}(z) - g_{X_N}(z)} \to 0$ in probability, for all $z$, as $N \to \infty$.
\end{lemma}
It is also useful to be able to assume that the entries of the random matrix are bounded. For this we have the following tool.

\begin{lemma}
Let $X_N$ be a symmetric random matrix, such that $Y_{ij}=\sqrt{N}X_{ij}$ are
i.i.d. for $i \leq j$ and all $N$. Suppose $\mathbb{E} Y_{ij}=0$ and $\mathbb{E} Y_{ij}^2=1$. Let
\begin{equation*}
\widehat{X}_{ij}=X_{ij}1_{\sqrt{N}|X_{ij}|<C}-\mathbb{E} (X_{ij}1_{\sqrt{N}%
|X_{ij}|<C}).
\end{equation*}
Then  for every $\varepsilon >0$, there exists $C$ that depends on $\varepsilon $, $z$, and the law of $Y_{ij}$ only, such that
\begin{equation*}
\mathbb{P}\{\abs{g_{\widehat{X}_N}(z) - g_{X_N}(z)}>\varepsilon \}<\varepsilon .
\end{equation*}
\end{lemma} 

These results are essentially corollaries of the Hoffman-Wielandt inequality, Lemma~\ref{hw ineq}.
Here is the proof of the result about truncation.

\begin{proof}
Clearly, it is enough to show that $\mathbb{E}\abs{g_{\widehat{X}_N}(z) - g_{X_N}(z)}^2$
can be made arbitrarily small uniformly in $N.$ By using the Hoffman-Wielandt inequality, we need to show that 
\begin{eqnarray*}
\frac{1}{N}\mathbb{E}\mathrm{Tr}\left( X_N-\widehat{X}_N\right)^2 &=&\frac{1}{N^2}\sum_{i,j=1}^{N}\mathbb{E}\left( \sqrt{N}X_{ij}-\sqrt{N}\widehat{X}_{ij}\right) ^2 \\
&=&\mathbb{E}\left( \sqrt{N}X_{12}-\sqrt{N}\widehat{X}_{12}\right) ^2
\end{eqnarray*}%
can be made arbitrarily small uniformly in $N$. However, 
\begin{equation*}
\mathbb{E}\left( \sqrt{N}X_{12}-\sqrt{N}\widehat{X}_{12}\right) ^{2}\leq
2\left\{ \left[ \mathbb{E}(Y_{12}1_{Y_{12}<C})\right] ^{2}+\mathbb{E}%
(Y_{12}1_{Y_{12}>C})^{2}\right\} .
\end{equation*}%
Since $\mathbb{E}Y_{12}=0$ and $\mathbb{E}Y_{12}^{2}=1,$ for every $\varepsilon >0,$ we can choose $C$ so large that 
$\left[ \mathbb{E}(Y_{12}1_{Y_{12}<C})\right] ^{2}<\varepsilon $ and $\mathbb{E}(Y_{12}1_{Y_{12}>C})^{2}<\varepsilon ,$ 
which completes the proof. \end{proof}

In light of these results, we will assume that $X_{ii} = 0$, i.e.
\begin{equation}
\label{gNz}
g_N(z) = \frac1N \sum_i \frac{1}{- z - x_i^T (X^{(i)} - z)^{-1} x_i}.
\end{equation}
In addition, we can and will assume that $|X_{ij}| < C.$

We would like to show (and this is the crux of the proof) that the right-hand side of (\ref{gNz}) approaches  $\frac{1}{-z-\ov g_N(z)}$ in probability, where $\ov g_N(z) = \BE g_N(z)$.

Write it as
\[
\frac1N \sum((-z-\ov g_N(z)) + (\ov g_N(z) - x_i^T(X^{(i)}-z)^{-1}x_i))^{-1}
\]
and note that, by the definition of the Stieltjes transform, the quantity $-z-\ov g_N(z)$ has nonzero imaginary part for all $z \in \BC \setminus \BR$.

It suffices to prove that
\[
\BP(\abs{\ov g_N(z) - x_i^T (X^{(i)} - z) x_i)} > \epsilon) < c_1 \exp(-c_2 N^\alpha)
\]
for some $c_1, c_2, \alpha > 0$. Recall that $g_N(z) = \frac1N \Tr((X-z)^{-1})$.
This is a consequence of the following two claims.

\begin{claim}\label{claim 1}
With high probability (at least $1-\exp(-cN^\alpha)$ for some $c, \alpha > 0$) we have 
\begin{equation}
\BP \{\abs{x_i^T (X^{(i)} - z)^{-1} x_i - \frac1N \Tr(X^{(i)}-z)^{-1}}>\varepsilon\}<c_1 \exp(-c_2(\varepsilon)N^\alpha).
\end{equation}
Here the expectation is conditional on the realization of $X^{(i)}$.
\end{claim}
\begin{claim}\label{claim 2}
With probability 1, 
\begin{equation}
\frac1N \abs{\Tr(X^{(i)}-z)^{-1} - \Tr(X-z)^{-1}} < \frac{c(z)}{N}.
\end{equation}
\end{claim}

\begin{proof}[Proof of claim \ref{claim 1}]
Let $B = (X^{(i)} - z)^{-1}$, $a = x_i$, and note that $a$ and $B$ are independent. Then
\[
\BE[a^T B a \vert B] = \BE[\sum a_i B_{ij} a_j \vert B] = \sum_i B_{ii} \BE a_i^2 = \frac1N \sum B_{ii} = \frac1N \Tr B.
\]
Next, we use concentration inequalities. The distribution of $\sqrt{N} X_{ij}$ is LSI. (It remains LSI even after truncation.) In addition the norm of the gradient of the quadratic form is bounded since the random variables $X_{ij}$ are bounded.  
Hence, the Herbst lemma is applicable and the conclusion follows. 
\end{proof}
\begin{remark}
The proof of this lemma can be done without invoking LSI, because this is a quadratic form and there are various techniques available for it (vaguely similar to the central limit theorems for sums of independent random variables). See example sheet 2.
\end{remark}

\begin{proof}[Proof of claim \ref{claim 2}]
The proof can be done using \emph{interlacing inequalities}: if $\lambda_1 \leq \dotsc \leq \lambda_N$ are eigenvalues of $X_N$, and $\mu_1 \leq \dotsc \leq \mu_{N-1}$ are eigenvalues of $X^{(i)}_N$, then
\[
\lambda_1 \leq \mu_1 \leq \lambda_2 \leq \dotsc \leq \mu_{N-1} \leq \lambda_N.
\]
\end{proof}

\newpage
\section{Gaussian ensembles}
Let $\xi_{ij}$, $\eta_{ij}$ be iid $\sim \CN(0,1)$ (i.e. iid standard normals). Let $X$ be an $N \times N$ matrix.
\begin{definition}
$X$ is called the Gaussian Unitary Ensemble ($\beta = 2$) if it is complex-valued, hermitian ($X_{ij} = \ov X_{ji}$), and the entries satisfy
\[
X_{ij} = \begin{cases}
\xi_{ii}, & i=j\\
\frac{1}{\sqrt{2}}(\xi_{ij} + \sqrt{-1}\eta_{ij}), & i<j
\end{cases}
\]
\end{definition}
\begin{definition}
$X$ is called the Gaussian Orthogonal Ensemble ($\beta = 1$) if it is symmetric and real-valued $(X_{ij} = X_{ji})$, and has
\[
X-{ij} = \begin{cases}
\sqrt{2}\xi_{ii}, & i=j\\
\xi_{ij}, & i<j
\end{cases}
\]
\end{definition}
There is also a Gaussian Symplectic Ensemble ($\beta = 4$) defined using quaternions, or using $2N \times 2N$ block matrices over $\BC$.

Let us compute the joint distribution of the GUE:
\begin{multline*}
\BP(X \in dx) =\\
c_N \exp\left(-\frac12 \sum_i x_{ii}^2 - \frac12 \sum_{i<j} (\Re x_{ij})^2 + (\Im x_{ij})^2\right) \prod_i dx_{ii} \prod_{i < j} d\Re x_{ij} d\Im x_{ij}\\
= c^{(2)}_N \exp(-\frac12 \Tr(X^2)) \prod_i dx_{ii} \prod_{i<j} d\Re x_{ij} d\Im x_{ij}.
\end{multline*}
We will refer to the volume element $d^{(N,2)}x$.

The main observation is that this is invariant under unitary transformations, because trace is. That is,
\begin{multline*}
\BP(UXU^* \in dx) = c^{(2)}_N \exp(-\frac12 \Tr(UXU^*)^2) d'^{(N,2)}x\\
= c^{(2)}_N \exp(-\frac12 \Tr X^2) d^{(N,2)}x = \BP(X \in dx),
\end{multline*}
where $d'^{(N,2)}x$ is induced by the transformation $X\rightarrow UXU^*$ from $d^{(N,2)}x$ and we used the fact that $d'^{(N,2)}x =d^{(N,2)}x$. (See proof of this claim in Section 5.2 of Deift's book  \cite{deift99}.) 

For the orthogonal ensemble ($\beta = 1$), we get
\[
\BP(X \in dx) = c^{(1)}_N \exp(-\frac14 \Tr X^2) d^{(N,1)}x,
\]
where $d^{(N,1)}x = \prod_{i \leq j} dx_{ij}$. For the symplectic ensemble ($\beta = 4$), we get
\[
\BP(X \in dx) = c^{(4)}_N \exp(-\Tr X^2) d^{(N,4)}x,
\]
Of course, the orthogonal, resp. symplectic, ensemble has distribution invariant to orthogonal, resp. symplectic transformations.

The reason we care about these ensembles is that for them we can compute the exact (non-asymptotic) distributions of eigenvalues.

\begin{theorem}
Suppose $X_N \in \CH^\beta$ for $\beta = 1, 2, 4$. Then the distribution of the eigenvalues of $X_N$ is given by
\[
p_N(x_1,\dotsc,x_N) = \ov c_N^{(\beta)} \prod_{i < j} \abs{x_i - x_j}^\beta \exp(-\frac{\beta}{4} \sum x_i^2) \prod dx_i
\]
where the constants $\ov c_N^{(\beta)}$ are normalization constants and can in theory be computed explicitly.
\end{theorem}
The proof will proceed by changing our basis from the matrix entries to the set of eigenvalues and eigenvectors, and then integrating out the component corresponding to the eigenvectors.

\begin{proof}
Write the (Hermitian) matrix $X$ as $X = U\Lambda U^*$, where $\Lambda$ is diagonal and $U$ is unitary; the columns of $U$ are the eigenvectors of $X$. When the eigenvalues of $X$ are are all distinct, this decomposition is unique up to (a) permuting the eigenvalues, and (b) multiplying $U$ by a diagonal matrix with entries $e^{i\theta_1}, e^{i\theta_2}, \dotsc, e^{i\theta_N}$.

Formally, let $F: U(N) \times \BR^N \to \CH^{(N)}$ be given by $F(U,\Lambda) = U\Lambda^* U$; then over points $X$ with distinct eigenvalues, the fibre of $F$ is $\BT^N \times S_N$. In particular, we can define a local isomorphism $\tilde F: (U(N) / \BT^N) \times (\BR^N / S_N) \to \CH^{(N)}$. (For an explicit example, take the eigenvalues to be increasing, and the first nonzero coordinate of each eigenvector to be real and positive.)

\begin{lemma}
The set of Hermitian $N \times N$ matrices with distinct eigencalues is open, dense, and has full measure.
\end{lemma}
The proof (esp. of the last result) is somewhat technical, and we won't give it. A good book to look is P. Deift, \emph{Orthogonal polynomials and random matrices: a Riemann-Hilbert approach} (AMS Courant Lecture Notes, 2000).

Let $\lambda_1,\dotsc,\lambda_N,p_1,\dotsc,p_{N^2-N}$ be local parameters on $(\BR^N / S_N) \times (U(N) / \BT^N)$. We would like to compute the Jacobian $\det(\frac{\del X_{ij}}{\del \lambda_\alpha}, \frac{\del X_{ij}}{\del p_\beta})$.

Let us formalize how we write $X$ as an element of $\BR^{N^2}$. Set
\[
\phi(X) = \begin{pmatrix} \frac{X_{11}}{\sqrt{2}} & \frac{X_{22}}{\sqrt{2}} & \dotsc & \frac{X_{11}}{\sqrt{2}} & \Re X_{12} & \Im X_{12} & \Re X_{13} & \Im X_{13} & \dotsc & \Im X_{N-1, N} \end{pmatrix}
\]
Note that $\abs{\phi(X)}^2 = \Tr(\frac12 X^2)$. Consequently, the transformation $L_U: \BR^{N^2} \to \BR^{N^2}$, $L_U(y) = \phi(U^* \phi^{-1}(y) U)$ is isometric (because conjugation by a unitary matrix preserves trace), i.e. $\det L_U = 1$. We will compute $\det(L_U(\frac{\del X_{ij}}{\del \lambda_\alpha}, \frac{\del X_{ij}}{\del p_\beta}))$, which equals $\det(\frac{\del X_{ij}}{\del \lambda_\alpha}, \frac{\del X_{ij}}{\del p_\beta})$.

Observe
\[
L_U\left(\frac{\del X}{\del \lambda_i}\right) = L_U\left(U\frac{\del \Lambda}{\del \lambda_i}U^*\right) = \frac{\del \Lambda}{\del \lambda_i}
\]
is the vector with all $N^2$ coordinates except the $i$th one equal to zero.

Next,
\[
L_U\left(\frac{\del X}{\del p_\beta}\right) = L_U \left(\frac{\del U}{\del p_\beta} \Lambda U^* + U \Lambda \frac{\del U^*}{\del p_\beta} \right) = U^* \frac{\del U}{\del p_\beta} \Lambda + \Lambda \frac{\del U^*}{\del p_\beta} U
\]
Recall that $U$ is unitary, i.e. $U^* U = I$. Differentiating this with respect to $p_\beta$ gives $\frac{\del U^*}{\del p_\beta}U + U^*\frac{\del U}{\del p_\beta} = 0$. Therefore, we can write
\[
L_U\left(\frac{\del X}{\del p_\beta}\right) = S_\beta \Lambda - \Lambda S_\beta, \quad S_\beta \equiv U^* \frac{\del U}{\del p_\beta}
\]
and
\[
\left(L_U\left(\frac{\del X}{\del p_\beta}\right)\right)_{ij} = (S_\beta)_{ij} (\lambda_j - \lambda_i).
\]

Therefore, $L_U(\frac{\del X_{ij}}{\del \lambda_\alpha}, \frac{\del X_{ij}}{\del p_\beta})$ looks like
\[
\begin{pmatrix}
I_N & 0 & 0 & 0 & \dotsc\\
0 & \Re(S_1)_{12}(\lambda_2-\lambda_1) & \Im(S_1)_{12}(\lambda_2 - \lambda_1) & \Re(S_1)_{13}(\lambda_3-\lambda_1) & \dotsc\\
0 & \Re(S_2)_{12}(\lambda_2-\lambda_1) & \Im(S_2)_{12}(\lambda_2 - \lambda_1) & \Re(S_2)_{13}(\lambda_3-\lambda_1) & \dotsc\\
\dotsc
\end{pmatrix}
\]
(the bottom left block of 0's comes from the derivatives of $\Lambda$ with respect to $p_\beta$, the top right block of 0's comes from the derivatives of $X_{ii}$ with respect to $p_\beta$, which get multiplied by $\lambda_i - \lambda_i$).

Computing the determinant, we will get
\[
\prod_{i < j}(\lambda_j - \lambda_i)^2 \det\begin{pmatrix}
\Re(S_1)_{12} & \Im(S_1)_{12} & \dotsc\\
\Re(S_2)_{12} & \Im(S_2)_{12} & \dotsc\\
\dotsc
\end{pmatrix}
\]
i.e. $\prod_{i < j}(\lambda_j - \lambda_i)^2 f(p_\beta)$ for some function $f$ that we don't care about.

Integrating out the $p_\beta$ gives density of eigenvalues of the GUE as
\[
p(\lambda) = c_N \prod_{i < j} (\lambda_i - \lambda_j)^2 \exp(-\frac12 \sum \lambda_i^2) d\lambda_1 \dotsc d\lambda_N
\]
\end{proof}

There are various directions in which the Gaussian ensembles can be generalized:
\begin{enumerate}
\item Unitary invariant ensembles:
\[
p_V(\lambda) = c_N \prod_{i < j} (\lambda_i - \lambda_j)^2 \exp(-\sum_{i=1}^N V(\lambda_i)) d^{(N)}\lambda
\]
where $V$ is (typically) a positive, convex, smooth function. This comes from an ensemble of Hermitian matrices with density $\exp(\Tr V(X)) d^{(N)} X$ as opposed to $\exp(-\frac12 \Tr(X^2)) d^{(N)}X$. The entries of these matrices aren't independent, but the distribution is invariant under unitary transformations.

Note that something like $\Tr^2(X)$ (in place of $\Tr(X^2)$) would also be unitarily invariant, but it's not clear that we can generalize the things that we will be saying about GUE to such objects.

\item $\beta$ ensembles:
\[
p^{(\beta)}(\lambda) = C_N \prod_{i < j} \abs{\lambda_i - \lambda_j}^\beta \exp(-\frac\beta4 \sum \lambda_i^2) d^{(N)}\lambda
\]
The cases $\beta=1$ (GOE) and $\beta=4$ (GSE) are reasonably similar to $\beta=2$ (GUE), but other values of $\beta$ are not. For a general $\beta>1$, it is possible to realize this distribution as the distribution of eigenvalues of certain random tridiagonal matrices with independent entries. (See the seminal paper by Dumitriu and Edelman \cite{dumitriu_edelman02}, and also papers by Edelman and Sutton [\href{http://arxiv.org/abs/math-ph/0607038}{arxiv:0607038}], Ramirez, Valk\'{o}, and Vir\'{a}g [\href{http://arxiv.org/abs/math/0607331v4}{arxiv:0607331v4}], and Valk\'{o}  and Vir\'{a}g [\href{http://arxiv.org/abs/0712.2000v4}{arxiv:0712.2000v4}] for more recent developments.) The tridiagonal matrices come from a standard algorithm for tridiagonalising Hermitian matrices; the fact that their entries end up being independent is really quite surprising!
\end{enumerate}

There is an interesting relation to statistical mechanics: we can write
\[
p^{(\beta)}(\lambda) = \frac1{Z_N} \exp(-\beta U(\lambda)), \quad U(x) = -\sum_{i<j} \log\abs{x_i-x_j} + \frac14 \sum x_i^2
\]
This looks like the Boltzmann distribution for a system of $N$ interacting electrically charged particles in the plane, which are confined to lie on the real axis and are placed in external potential $ x^2$. In particular, this gives intuition for eigenvalues ``repelling'' each other.

Let us now derive the results for GUE, which is the simplest case of the above ensembles. Recall the Vandermonde determinant:
\[
\prod_{i<j} (x_i - x_j)^2 =
\begin{vmatrix}
1 & \dotsc & 1\\
x_1 & \dotsc & x_N\\
\vdots & \ddots & \vdots\\
x_1^{N-1} & \dotsc & x_N^{N-1}
\end{vmatrix}^2
=
\begin{vmatrix}
p_0(x_1) & \dotsc & p_0(x_N)\\
p_1(x_1) & \dotsc & p_1(x_N)\\
\vdots & \ddots & \vdots\\
p_{N-1}(x_1) & \dotsc & p_{N-1}(x_N)
\end{vmatrix}^2
\]
where $\abs{\cdot}$ denotes the determinant of a matrix. Here, $p_j$ are monic polynomials of degree $j$; we will use (up to a factor) the Hermite polynomials, which are orthogonal with respect to a convenient measure.

In particular, we can write the density of eigenvalues as
\[
p(x_1,\dotsc,x_N) = c_N
\begin{vmatrix}
p_0(x_1) e^{-\frac14 x_1^2} & \dotsc & p_0(x_N) e^{-\frac14 x_N^2}\\
p_1(x_1) e^{-\frac14 x_1^2} & \dotsc & p_1(x_N) e^{-\frac14 x_N^2}\\
\vdots & \ddots & \vdots\\
p_{N-1}(x_1) e^{-\frac14 x_1^2} & \dotsc & p_{N-1}(x_N) e^{-\frac14 x_N^2}
\end{vmatrix}^2
= \tilde c_N \left(\left.\det (A_{ij})\right|_{i,j=1}^N \right)^2 = \tilde c_N \det A^T A
\]
where the entries of $A$ are given by $A_{ij} = a_{i-1} P_{i-1}(x_j) e^{-\frac14 x_j^2} \equiv \phi_{i-1}(x_j)$, for $\phi_i$ the Hermite functions. $a_i P_i(x)$ are the normalized Hermite polynomials, i.e.
\[
\int_\BR a_i P_i(x) a_j P_j(x) e^{-x^2/2} dx = \delta_{ij}
\]
(i.e. they are orthonormal, but with respect to the measure $e^{-x^2/2} dx$).

Let us now compute the entries of $A^t A$:
\[
(A^T A)_{ij} = \sum_{k=1}^N A_{ki} A_{kj} = \sum_{k=0}^{N-1} \phi_k(x_i) \phi_k(x_j).
\]
Let
\[
K_N(x,y) = \sum_{k=0}^{N-1} \phi_k(x) \phi_k(y), \quad \text{the Christoffel-Darboux kernel},
\]
then we have
\[
p_N(x_1,\dotsc,x_N) = \tilde c_N \det(K_N(x_i,x_j)).
\]

So far we haven't used any property of the Hermite polynomials. Suppose, however, that we want to compute marginals of this probability distribution, for example, the distribution of $x_1$ alone. In that case, we would need to integrate out the dependence on all the other $x_i$.

The orthonormality of the (normalized) Hermite polynomials implies that if the kernel $K_N(x,y)$ is defined using these polynomials, then it satisfies the following condition:
\begin{equation}
\label{kernel_formula1}
\int_{\BR} K_N(x,y) K_N(y,z) dy = K_N(x,z).
\end{equation}

\begin{remark}
If $K_N$ is considered as an operator on functions given by
\[ 
(K_N f)(x) = \int K_N(x,y) f(y) dy,
\]
then $K_N$ is the orthogonal projection onto the space spanned by the first $N$ Hermite functions. The identity then simply says that the orthogonal projection operator is an idempotent, that is, it squares to itself.
\end{remark}
Let
\[
J_N = (J_{ij})_{1 \leq i,j \leq N} = (K(x_i,x_j))_{1 \leq i,j \leq N}
\]
for some kernel $K$ (not necessarily the Christoffel-Darboux one).  Then we have the following result.
\begin{theorem}[Mehta]
Assume that $K$ satisfies 
\begin{equation}
\label{kernel_formula2}
\int_{\BR} K(x,y) K(y,z) dy = K(x,z).
\end{equation}
Then
\begin{equation}
\label{Mehta_formula}
\int_\BR \det(J_N) d\mu(x_N) = (r-N+1) \det(J_{N-1}),
\end{equation}
where $r = \int K(x,x) d\mu(x)$.
\end{theorem}

\begin{proof}
Expand the left-hand side:
\begin{align*}
\int_\BR \det(J_N) d\mu(x_N) &= \int_\BR \sum_{\sigma \in S_N} \sgn(\sigma) K(x_1,x_{\sigma(1)})\dotsc K(x_N,x_{\sigma(n)}) d\mu(x_N)\\
&= \int_\BR \sum_{k=1}^N \sum_{\sigma: \sigma(N) = k} \sgn(\sigma) K(x_1,x_{\sigma(1)})\dotsc K(x_N,x_k) d\mu(x_N)\\
\end{align*}
We now split the sum into two cases: $k=N$ and $k \neq N$.

When $k=N$, we straightforwardly get $r \det J_{N-1}$, since $\sigma$ is essentially running over all permutations in $S_{N-1}$, and the sign of $\sigma$ as a permutation in $S_N$ and in $S_{N-1}$ is the same.

When $k < N$, let $j = \sigma^{-1}(N)$, and let $\hat \sigma \in S_{N-1}$ be given by
\[
\hat \sigma(i) = \begin{cases}\sigma(i), & i \neq j\\
k, & i=j
\end{cases}
\]
It's easy to check that the map $\{\sigma \in S_N: \sigma(N) = k\} \to S_{N-1}$, $\sigma \mapsto \hat\sigma$, is a bijection; and moreover, $\sgn(\hat\sigma) = -\sgn(\sigma)$ because essentially they differ by a transposition $(kN)$. We now write
\begin{multline*}
\int_\BR \sum_{\sigma: \sigma(N) = k} \sgn(\sigma) K(x_1,x_{\sigma(1)})\dotsc K(x_N,x_k) d\mu(x_N)\\
= \int_\BR \sum_{\sigma: \sigma(N) = k} \sgn(\sigma) K(x_1,x_{\sigma(1)})\dotsc K(x_{N-1},x_{\sigma(N-1)}) K(x_j, x_N) K(x_N,x_k) d\mu(x_N)\\
= \sum_{\hat \sigma \in S_{N-1}} -\sgn(\hat \sigma) K(x_1,x_{\hat\sigma(1)})\dotsc K(x_{N-1}, x_{\hat\sigma(N-1)})\\
=-\det J_{N-1}
\end{multline*}
where the second-to-last line uses condition (\ref{kernel_formula2}). Since we get the same answer for all $k=1,\dotsc,N-1$, the statement of the theorem follows.
\end{proof}

\begin{corollary} For the density of eigenvalues of the GUE ensemble we have
\[
\int p_N(x_1,\dotsc,x_N)dx_{k+1}\dotsc dx_N = \frac{(N-k)!}{N!} \det(K_N(x_i,x_j)_{1 \leq i,j \leq k})
\]
\end{corollary}
\begin{proof} We have
\[
p_N(x_1,\dotsc,x_N) = \tilde c_N \det(K_N(x_i,x_j)).
\]
Hence,
\[
\int_\BR p_N(x_1,\dotsc,x_N) dx_N = \tilde c_N \det(K_N(x_i,x_j)_{1 \leq i,j \leq N-1})
\]
(since $r=N$ for our kernel $k_N$ by the orthonormality of Hermite functions);
\[
\int_\BR\int_\BR p_N(x_1,\dotsc,x_N) dx_{N-1} dx_N =  2 \tilde c_N \det(K_N(x_i,x_j)_{1 \leq i,j \leq N-2})
\]
(we still have $r=N$, since the kernel is the same, but we've reduced the size of the matrices from $N$ to $N-1$, so the factor in the previous theorem is now equal to 2). Continuing on in this manner,
\[
\int p_N(x_1,\dotsc,x_N)dx_1\dotsc dx_N = N! \tilde c_N,
\]
implying that $\tilde c_N = 1/N!$; and in general, 
\[
\int p_N(x_1,\dotsc,x_N)dx_{k+1}\dotsc dx_N = \frac{(N-k)!}{N!} \det(K_N(x_i,x_j)_{1 \leq i,j \leq k}).
\]
\end{proof}

\begin{definition}
A \emph{correlation function} is
\[
R_k(x_1,\dotsc,x_k) = \frac{N!}{(N-k)!} \int p_N(x_1,\dotsc,x_N) dx_{k+1},\dotsc,dx_N.
\]
\end{definition}
Thus, the density of eigenvalues of the GUE has correlation function(s) $R_k = \det(K_N(x_i,x_j)_{1 \leq i,j \leq k})$.

Let's expand a few of these to check that they are actually computable:
\[
R_1(x) = K_N(x,x) = \sum_{k=0}^{N-1} \phi_k(x,x)
\]
and
\[
R_2(x) = \begin{vmatrix}
K_N(x,x) & K_N(x,y)\\
K_N(x,y) & K_N(x,x)
\end{vmatrix}
\]
(where we've used the symmetry of the kernel).

What is the meaning of correlation functions? Let $B$ be a Borel set. Then,
\begin{align*}
\int_B R_1(x) dx &= \BE[\#\{\text{eigenvalues in $B$}\}]\\
\iint \limits_{B \times B} R_2(x,y) dx dy &= \BE[\#\{\text{ordered pairs of eigenvalues in $B$}\}]
\end{align*}

If eigenvalues are considered as random points then their collection is a random point process. The previous equations say that the intensities of this process are given by the correlation functions $R_k$. Point processes for which intensities (defined in terms of expectations of the number of points, pairs of points, etc.) are given by $\det K$ for some kernel $K$ are called \emph{determinantal point processes}. There exist explicit conditions on a kernel $K$ that establish when determinants constructed from $K$ can generate the intensities of a valid point process. We have effectively shown that any $K$ that can be written as $K(x,y) = \sum f_k(x) f_k(y)$, where $f_k$ are orthonormal functions with respect to some measure, gives rise to a determinantal point process. For more, see Soshnikov \cite{soshnikov00b} or Peres \cite{hkpv06}.

\subsection{Other statistics of GUE}
In addition to the correlation functions (intensities), we will now look at
\[
A_m(I) = \BP(\text{exactly $m$ eigenvalues in $I$})
\]
for some interval $I$. For example,
\[
A_0(I) = \BE \prod_{i=1}^N (1-\one_I (\lambda_i)).
\]

Define
\[
F_I(t) = \BE \prod_{i=1}^N (1-t \one_I(\lambda_i))
\]
Then it is easy to see $F_I(1) = A_0(I)$, $-F_I'(1) = A_1(I)$, and more generally
\[
A_m(I) = \frac{(-1)^m}{m!} F_I^{(m)}(1).
\]
Therefore, we now focus on computing $F_I$.

Expanding the product,
\begin{align*}
F_I(t) &= 1 - t\BE \sum_i \one_I(\lambda_i) + t^2 \BE \sum_{i \leq j} \one_I(\lambda_i) \one_I(\lambda_j) - \dotsc\\
&= 1 - t \int_I R(x_1) dx_1 + \frac{t^2}{2} \int_{I \times I} R_2(x_1,x_2) dx_1 dx_2 - \frac{t^3}{3!} \int_{I^3} R_3(x_1,x_2,x_3) dx_1 dx_2 dx_3 + \dotsc\\
&= 1 + \sum_{k=1}^\infty \frac{(-t)^k}{k!} \int_{I^k} \det(K_N(x_i,x_j)_{1 \leq i,j \leq k})\\
&\equiv \det(I-t K_N(x,y))
\end{align*}
where we take the last equality to be the definition of the \emph{Fredholm determinant} of the operator with kernel $K_N(x,y)$ acting on $L^2(I)$.

\begin{remark}
An aside on operators and Fredholm determinants:

An operator with kernel $K(x,y)$ acts on $f$ as $(Kf)(x) = \int K(x,y) f(y) dy$. The formula above restricts to the usual notion of determinant when $K$ is a matrix: $(Kv)_x=\sum_{y=1}^N K_{xy} v_y$. (This is an exercise).

When $K$ is a trace class operator (i.e. $Tr(K) = \int K(x,x) dx$ exists and is finite), we have other formulas for the Fredholm determinant.
\begin{enumerate}
\item
\[
\det(I-K) = \exp(-\sum_{m=1}^\infty \frac1m \Tr(K^m))
\]
where $K^m$ denotes the result of applying $K$ $m$ times (i.e. convolution, not product), and the trace of a kernel is given by $Tr(K) = \int K(x,x) dx$. Note that if $K=\lambda$, i.e. the operator acts as multiplication by $\lambda$ on $\BR$, we have $\exp(-\frac1m \lambda^m) = \exp(\log(1-\lambda)) = 1-\lambda$ as expected.
\item
\[
\det(I-K) = \prod(1-\lambda_j)
\]
the product being taken over all the eigenvalues of $K$ counted with algebraic multiplicity.
\end{enumerate}

We will not prove the equivalence between these definitions; it's a hard theorem in analysis. For more information on the Fredholm determinants, a good source is Section XIII.17 in the 4th volume of Reed and Simon \cite{reed_simon78}. In particular, see formula (188) on p. 323 and the classical formula on p.362.)
\end{remark}

\subsection{Asymptotics}
What is the scaling limit of $K_N(x,y)$ as $N \to \infty$? That is, if we rescale and center appropriately, what sort of functional dependence will we get? Recall
\[
K_N(x,y) = \sum_{j=1}^N \phi_j(x) \phi_j(y)
\]
where $\phi_j$ are (up to a factor) the Hermite functions. For convenience, let us use the functions orthonormal with respect to weight $e^{-t^2}$ instead of $e^{-t^2 /2}$. This corresponds to dividing the GUE matrix by $\sqrt{2}$ and has the advantage that $\phi_j(x)$ can be written in terms of the standard Hermite polynomials: 
\begin{equation*}
\phi_j(x) = H_n(x)e^{-t^2/2}\frac{1}{\pi^{1/4} 2^{N/2} \sqrt{N!}}.
\end{equation*}
 The Christoffel-Darboux formula for orthogonal polynomials gives
\[
\sum_{j=1}^N \phi_j(x) \phi_j(y) = \sqrt{N/2} \frac{\phi_{N-1}(x) \phi_N(y) - \phi_N(x) \phi_{N-1}(y)}{x-y}
\]
The asymptotics for individual Hermite functions are well-known (see Chapter VIII in Szeg\"o \cite{szego67}):
\begin{align*}
&\lim (-1)^m m^{1/4} \phi_{2m} \left(\frac{\pi \xi}{\sqrt{2m}} \right) = \frac{1}{\sqrt{\pi}} \cos(\pi \xi)\\
&\lim (-1)^m m^{1/4} \phi_{2m+1} \left(\frac{\pi \eta}{\sqrt{2m}} \right) = \frac{1}{\sqrt{\pi}} \sin(\pi \eta).
\end{align*}
Therefore, we conclude
\[
\lim_{N \to \infty} \frac{\pi}{\sqrt{2N}} K_N\left(\frac{\pi\xi}{\sqrt{2N}}, \frac{\pi\eta}{\sqrt{2N}} \right) = \frac{\sin(\pi(\xi-\eta))}{\pi(\xi-\eta)} \equiv K_{\text{bulk}}(\xi,\eta)
\]
(The extra $(\sqrt{2N})^{-1/2}$ is there to compensate for the $\xi-\eta$ in the denominator.) Note that the $N$ eigenvalues of the GUE lie on $[-\sqrt{2N},  \sqrt{2N}]$ because we haven't rescaled the matrix entries. Here we are looking at points that are $O(1/\sqrt{N})$ apart, i.e. under this scaling the distance between adjacent eigenvalues remains roughly constant when $N$ increases. 

If we are interested in the distribution of the number of eigenvalues on an interval, we will be interested in $K_{\text{bulk}}(\xi,\xi)$ (i.e., $\lim_{\eta \to \xi} K_{\text{bulk}}(\xi,\eta)$), which is 1: i.e. under this rescaling the expected number of eigenvalues in an interval is simply its length.

If we want to compute the correlation functions, we will need to evaluate the determinants: e.g.,
\[
\lim_{N \to \infty} \left(\frac{\pi}{\sqrt{2N}} \right)^2 R_2\left(\frac{\pi\xi}{\sqrt{2N}}, \frac{\pi\eta}{\sqrt{2N}} \right) = 1 - \left( \frac{\sin(\pi(\xi-\eta))}{\pi(\xi-\eta)} \right)^2 \equiv R_{2,\text{bulk}}(\xi,\eta).
\]
For example,
\[
\sum_{i,j} f(\sqrt{N} \lambda_i, \sqrt{N} \lambda_j) \to \int_{\BR^2} f(\xi,\eta) R_{2,\text{bulk}}(\xi,\eta)d\xi d\eta
\]
provided $f \to 0$ sufficiently quickly at infinity (e.g., if $f$ has compact support).

\begin{remark}
$K_{\text{bulk}}$ acts particularly simply on Fourier transforms: $\hat{K_{\text{bulk}}f} = \one_{[-\pi,\pi]} \hat f$. It's not entirely surprising that we should get a projection operator in the limit (each $K_N$ was a projection operator), although it's not clear why it's such a simple projection operator.
\end{remark}

\subsection{Scaling at the edge of the spectrum}
The calculations we did above applied to the eigenvalues near 0. We might also be interested in the distribution of the eigenvalues near the edge of the spectrum, e.g. in the distribution of the largest eigenvalue.

The coordinate change we will be using is $x = \sqrt{2N} + \frac{t}{\sqrt{2}N^{1/6}}$ (where other books might use $ N^{1/6}$ or $2 N^{1/6}$ instead, depending on how the matrix entries were scaled to begin with). In this case, the asymptotics for the Hermite functions are
\[
\phi_N(x) = \pi^{1/4} 2^{N/2+1/4} (N!)^{1/2} N^{-1/12} \left( \Ai(t) + O(N^{-2/3}) \right),
\]
where $\Ai(t)$ is the \emph{Airy function}: it satisfies $y'' = ty$ and $y(+\infty) = 0$ (plus some normalization constraint, since this is a second-order ODE). It can also be defined as a contour integral
\[
\Ai(t) = \frac{1}{2\pi i} \int \exp(-tu + u^3/3) du
\]
where the integral is over two rays $r (e^{-\pi/3})$ from $\infty$ to $0$, then $r e^{\pi/3}$ from $0$ to $\infty$. (This integral naturally arises when the asymptotics of Hermite polynomials are derived through the steepest descent formula. The Airy function can be also defined as a particular case of Bessel functions, or as a real integral of $\cos(tu + u^3/3)$. Note, however, that this integral is difficult to compute since it has a singularity as $u \to \infty$.)

\begin{figure}[htbp]
\includegraphics[width=8cm]{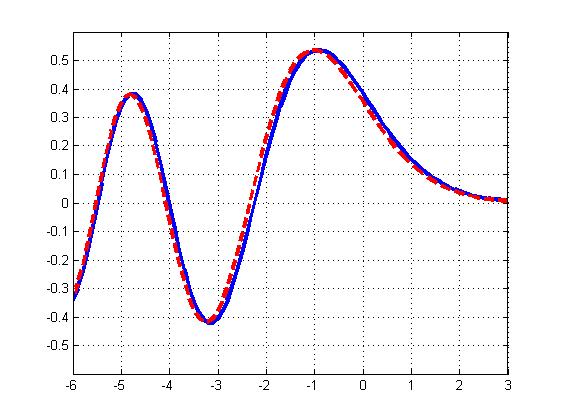} \label{airy_function}
\caption{The scaled Hermite function (solid line, N=100) and the Airy function (dashed line)}
\end{figure}

\begin{theorem}[Forrester 1993]
Let $x = \sqrt{2N} + \xi 2^{-1/2}N^{-1/6}$, $y = \sqrt{2N} + \eta 2^{-1/2} N^{-1/6}$, then
\[
\lim_{N \to \infty} \frac{1}{2^{1/2}N^{1/6}} K_N(x,y) = \frac{\Ai(\xi) \Ai'(\eta) - \Ai(\eta) \Ai'(\xi)}{\xi - \eta} \equiv K_{\text{edge}}(\xi,\eta).
\]
\end{theorem}
This result originally appeared in \cite{forrester93} Here it is no longer obvious what $K_{\text{edge}}(\xi,\xi)$ is, but after massaging the differential equation for $\Ai$ we conclude
\[
K_{\text{edge}}(\xi,\xi) = -\xi (\Ai(\xi))^2 + (\Ai'(\xi))^2.
\]

\begin{figure}[htbp]
\includegraphics[width=8cm]{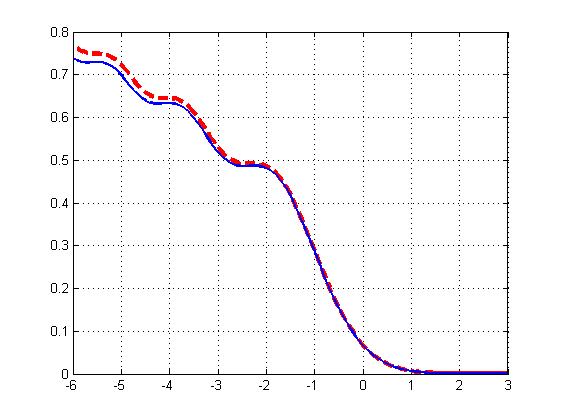} \label{airy_kernel}
\caption{The scaled Christoffel-Darboux kernel at the edge (solid line, N=100) and the Airy kernel (dashed line)}
\end{figure}

We could now derive an expression for the correlation functions, but it won't be anything nice.

People often care more about the Wishart ensemble, i.e., the matrices of the form $X^t X$ where $X$ are complex or real Gaussian matrices. In that case, the story is similar but we use Laguerre polynomials instead of Hermite ones.

Suppose that we want to know the distribution of the largest eigenvalue $\lambda_{\max}$. Then
\begin{multline*}
\BP(\lambda_{\max} < 2\sqrt{N} + \xi N^{-1/6}) = \BP(\text{no eigenvalues in }[2\sqrt{N} + \xi N^{-1/6}, \infty))\\
= A_0[2\sqrt{N} + \xi N^{-1/6}, \infty) = \det (I - K_N)
\end{multline*}
(the Fredholm determinant), where $K_N$ is a kernel operator on $L^2[2\sqrt{N} + \xi N^{-1/6}, \infty)$. It is plausible (although we won't go into the technical details) that this should converge, after the appropriate change of variables, to
\[
\det(I-K_{\text{edge}})
\]
where $K_{\text{edge}}$ is the operator with kernel $K_{\text{edge}}(x,y)$ acting on $L^2[\xi,\infty)$. This sort of thing can be tabulated, and is called the Tracy-Widom distribution for $\beta = 2$. It turns out (Bornemann, \href{http://arxiv.org/pdf/0904.1581v5.pdf}{arxiv:0904.1581}) that for numerical computations the integral operator can be approximated by a sum, which makes the computation manageable. An alternative method uses differential equations developed by Tracy and Widom (\href{http://arxiv.org/pdf/hep-th/9211141.pdf}{arxiv:hep-th/9211141}), which are related to Painlev\'e differential equations.

\subsection{Steepest descent method for asymptotics of Hermite polynomials}
The asymptotics for Hermite, Laguerre and other classical orthogonal polynomials are derived using the steepest descent method.
Morally, steepest descent is a way to compute the asymptotics for contour integrals of the form
\[
\int_C e^{tf(z)}dz
\]
as $t \to \infty$. The idea is to change the contour so that it passes through the critical points where $f'(z) = 0$; with luck, the portions of the contour near these points give the largest contribution to the integral.

More precisely: Let $z_0$ be s.t. $f'(z_0) = 0$, and write
\[
f(z) = f(z_0) - \frac12 A(e^{i\theta/2}(z-z_0))^2 + \dotsc
\]
Then necessarily $z_0$ is a saddle point for $\abs{f(z)}$. Change the contour to go along the line of steepest descent through $z_0$. It is then plausible that the contour integral away from $z_0$ is negligibly small.

Formally, we parametrize $u = e^{i\theta/2}(z-z_0)$, then
\[
\int_C e^{tf(z)}dz \approx e^{tf(z_0)} e^{-i\theta/2}\int_C e^{-tA/2 u^2} du \approx e^{tf(z_0)} e^{-i\theta/2} \sqrt{\frac{2\pi}{tA}}
\]
for large $t$, because this is essentially the Gaussian integral. Recalling that $-Ae^{i\theta}$ is the second derivative, we have
\[
\int_C e^{tf(z)}dz \approx e^{tf(z_0)} \sqrt{\frac{2\pi}{-t f''(z_0)}}.
\]
This approximation will hold if the original contour integral can be approximated by integrals in the neighborhood of critical points.

The steepest descent calculations can be carried through identically for asymptotics of $\int e^{tf(z)}h(z) dz$, where $h$ is sufficiently smooth near the critical points; we will simply pick up an extra factor:
\[
\int_C h(z) e^{tf(z)}dz \approx h(z_0) e^{tf(z_0)} \sqrt{\frac{2\pi}{-t f''(z_0)}}.
\]
For a more detailed and rigorous discussion see Chapter 5 in de Bruijn's book \cite{bruijn58}, or Copson \cite{copson67}, or any other text on asymptotic methods.

For Hermite polynomials, we have the recurrence relation
\[
x H_n(x) = \frac12 H_{n+1}(x) + n H_{n-1}(x)
\]
\begin{remark}
All families of orthogonal polynomials have some three-term recurrence of this form (with some function $C(n)$ instead of $n$ above); however, usually it is hard to compute explicitly. It is the fact that we can do it for the classical polynomials (like Hermite, Laguerre, Jacobi) that makes the steepest descent method possible. 
\end{remark}

Let
\[
g(t) = \sum_{k=0}^\infty \frac{H_k(x)}{k!} t^k
\]
be the generating function of the Hermite polynomials, then the above recurrence gives
\[
g'(t) = (2x-2t)g(t) \implies g(t) = e^{2xt - t^2}.
\]
 Since the Hermite polynomials are coefficients of the Taylor expansion of this, we have (by the residue formula)
\[
H_n(x) = \frac{n!}{2\pi i} \int_{\abs{z}=1} \frac{e^{2xz - z^2}}{z^{n+1}}dz = \frac{n!}{2\pi i} \int_{\abs{z}=1} \exp(2xz - z^2 - (n+1)\log z)dz
\]
Let $\tilde z = z / \sqrt{2n}$, $y = x / \sqrt{2n}$, then the above integral becomes
\[
c_n \int_{\abs{\tilde z} = 1} e^{nf(\tilde z)} \frac{d\tilde z}{\tilde z}, \quad f(\tilde z) = 4y\tilde z - 2\tilde z^2 - \log \tilde z
\]
We can change the contour to be $\abs{\tilde z} = 1$, because the only singularity is at the origin. 
We have
\[
f'(\tilde z) = 4y - 4\tilde z  - \frac{1}{\tilde z},
\]
so the critical points satisfy the equation
\[
f'(\tilde z_0) = 0 \text{, that is, } \tilde z_0^2 - y \tilde z_0 + \frac14 = 0
\]
This gives rise to three different asymptotics: $\abs{y} > 1$ (the two real roots case), $\abs{y} < 1$ (the two complex roots case), and $y = \pm1$  (the double root case).

The case of two real roots gives something exponentially small, and we are not much interested in it here. The case of $y = 1$ gives $f''(z) = 0$, so steepest descent as we described it is not applicable. Instead, there is a single cubic singularity, and changing a contour through it appropriately will give the Airy function asymptotics.

If $\abs{y} < 1$ and we have two complex critical points $z_0 = \frac12 (y \pm i \sqrt{1-y^2}) = \frac12 e^{\pm i \theta_c}$, we deform the contour as below:
\begin{figure}[ht]
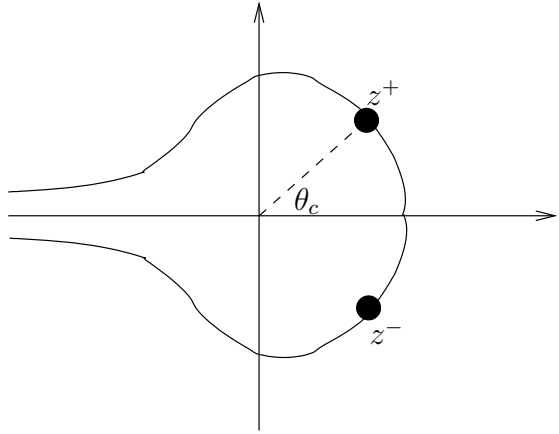
\caption{Deformed contour going through the two critical points}
\end{figure}

The steepest descent formula will have two terms, one for each critical point. At criticality we have
\[
f(z_0) = y^2 \pm iy\sqrt{1-y^2} +\frac12 + \log2 \mp i\theta_c
\]
and therefore,
\[
H_n(\sqrt{2n} y) \approx c_n \left( \frac{e^{ny^2} \exp(in(y \sqrt{1-y^2} - \theta_c))}{\sqrt{-nf''(z^+)}} + \frac{e^{ny^2} \exp(-in(y \sqrt{1-y^2} - \theta_c))}{\sqrt{-nf''(z^-)}} \right)
\]
where $\theta_c = \cos^{-1}(y)$; rewriting, we get
\[
H_n(\sqrt{2n} y) \approx c_n e^{ny^2} \cos\left(n(y\sqrt{1-y^2} - \theta_c) + \theta_0\right)
\]
(where the $\theta_0$ comes from $\sqrt{-n f''(z^\pm)}$).

The asymptotic regime we were considering was $\sqrt{2n} y = \frac{\xi}{\sqrt{2n}}$. This gives $y = \xi/(2n)$ very small. In particular, $e^{ny^2} = e^{\xi^2/4n} = 1+\frac{\xi^2}{4n} + O(n^{-2})$, $y\sqrt{1-y^2} = \frac{\xi}{2n} + O(n^{-2})$, and $\theta_c = \cos^{-1}(y) \approx \frac{\pi}{2} - \frac{xi}{2n} + O(\frac{1}{n^2})$. We can also check that $\theta_0 = O(n^{-1})$. Consequently,
\[
H_n\left(\frac{\xi}{\sqrt{2n}}\right) = c_n e^{\xi^2/4n} \cos(\xi - n\pi/2 + O(n^{-1}))
\]
which is the asymptotic we had before. (The exponential factor upfront is due to the fact that here we were approximating Hermite \emph{polynomials} rather than Hermite \emph{functions}.)

\newpage
\section{Further developments and connections}
\subsection{Asymptotics of invariant ensembles (Deift et al.)}
Suppose the eigenvalue distribution satisfies $p(\lambda_1,\dotsc,\lambda_N) = \Delta(\lambda)^2 \exp(-\sum_{i=1}^N V(\lambda_i))$ for some potential function $V$. Here, $\Delta$ is the Vandermonde determinant we had before. The asymptotics for this distribution would follow from asymptotics for a family of orthogonal polynomials with weight $\exp(-V(x))dx$.

The difficulty is that there is no explicit formula for the coefficients in the three-term recurrence relation for these polynomials.

Deift et al. found a way to find these asymptotics (see Chapter 7 in \cite{deift99} or Section 6.4 in Kuijlaars' review \href{http://arxiv.org/abs/1103.5922}{arxiv:1103.5922}); the solution method is related to a multidimensional version of the Riemann-Hilbert problem.

The Riemann-Hilbert problem is as follows. Consider a contour $\Sigma$ on $\BC$. We would like to find two functions $Y_\pm(z): \BC \to \BC^n$ (for the classical Riemann-Hilbert problem, $n=1$), which are analytic on the two regions into which $\Sigma$ partitions $\BC$, and with the property that $Y_+(z) = V(z) Y_-(z)$ on $\Sigma$, for some matrix $V$. (There will also be some conditions on the behaviour of $Y$ at infinity.) It turns out that for $n>1$ this can be constructed so that one of the components of $Y$ gives the orthogonal polynomials we want. Deift et al. found a way to analyze the asymptotics of these solutions.

\subsection{Dyson Brownian motion}
Suppose that the matrix entries $X_{ij}$ follow independent Ornstein-Uhlenbeck processes.
\begin{remark}
An Ornstein-Uhlenbeck process satisfies the stochastic differential equation (SDE)
\[
dx_t = -\theta (x_t-\mu) dt + \sigma dW_t,
\]
where $W$ is a Brownian motion. It has a stationary distribution, which is Gaussian centered on $\mu$ with variance $\frac{\sigma^2}{2\theta}$. Such a process can be used to model, e.g., the relaxation of a spring in physics (normally exponential) in the presence of thermal noise. We will undoubtedly assume $\mu = 0$ (i.e., $X_{ij}$ centered on 0), and possibly $\sigma = 1$.
\end{remark}
The eigenvalues, being differentiable functions of the matrix entries $X_{ij}$, will then also follow a diffusion process. Applying It\^o's formula (nontrivially), this can be found to be
\[
d\lambda_i = \frac{1}{\sqrt{N}} dB_i + \left(-\frac{\beta}{4} \lambda_i + \frac{\beta}{2N} \sum_{i\neq j} \frac{1}{\lambda_i - \lambda_j} \right), \quad i=1,\dotsc,N
\]
Here, $B_i$ are Brownian motions, and $\beta = 1,2,4$ according to whether the ensemble is orthogonal, unitary, or symplectic. The best introduction is Dyson's original paper \cite{dyson62}. A system of coupled SDE's is not trivial to solve, but it turns out that the distribution of the solution converges to an equilibrium distribution. One could ask, e.g., about the speed of convergence. This material is covered in L\'aszl\'o Erd\"os's lecture notes, \url{http://www.mathematik.uni-muenchen.de/~lerdos/Notes/tucson0901.pdf}.

\subsection{Connections to other disciplines}
The methods that were initially developed for random matrices have since been used in various other places.
\subsubsection{Longest increasing sequence in a permutation (Baik-Deift-Johansson)}
Consider a permutation in $S_n$, e.g. $\pi \in S_4$ which maps $1234$ to $1324$. We define $l(\pi)$ to be the length of the longest increasing subsequence in $\pi$; here, $l(\pi) = 3$ (corresponding to subsequences $134$ or $124$). We would like asymptotics for $l(\pi)$, and particularly for its distribution, as $n \to \infty$.

We will use the \emph{RSK (Robinson-Schensted-Knuth) correspondence} between permutations and pairs of \emph{standard Young tableaux}. A standard Young tableau is the following object. First, partition $n$, i.e. write it as $n = \lambda_1 + \lambda_2 + \dotsc + \lambda_r$, where $\lambda_i$ are integers, and $\lambda_1 \geq \lambda_2 \geq \dotsc \geq \lambda_r > 0$. E.g., for $n=10$ we might write $10 = 5 + 3 + 2$. We then draw the corresponding shape, where the $\lambda_i$ are row lengths:
\begin{figure}[ht]
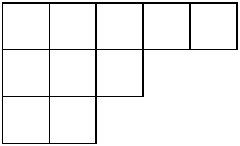
\end{figure}

To turn this into a standard Young tableau, we will fill it with numbers $1,\dotsc,n$ which increase along rows and along columns. The RSK correspondence asserts that there is a bijection between permutations $\pi$ and pairs of standard Young tableaux of the same shape. For example, there is a permutation $\pi \in S_{10}$ corresponding to the pair
\[
P = \begin{matrix}
1 & 2 & 5 & 6 & 8\\3 & 4 & 7\\9 & 10
\end{matrix}
\qquad\qquad
Q = \begin{matrix}
1 & 3 & 5 & 7 & 9\\2 & 4 & 6\\8 & 10
\end{matrix}
\]
The length $l(\pi)$ is the length of the top row of the standard Young tableau. We define also $r(\pi)$ the number of columns in the standard Young tableau; this is the length of the longest decreasing subsequence in $\pi$.

Standard Young tableaux have been well-studied because of their relationship to representations of $S_n$. See for example W. Fulton. \emph{Young Tableaux}. Cambridge University Press 1997.

To get the distribution of $l(\pi)$, we need to count the number of standard Young tableaux with a given length of the top row. Now, there is a formula, due to Frobenius and Young, for the number $f(\lambda)$ of standard Young tableaux of a given shape $\lambda$ (this is, by the way, the dimension of the corresponding irreducible representation of $S_n$). (Here, $\lambda$ is a partition of $n$, that is, $n = \lambda_1 + \dotsc + \lambda_r$ with $\lambda_1 \geq \lambda_2 \geq \dotsc \geq \lambda_r > 0$.)

Let $h_i = \lambda_i + (r-i)$, then the Frobenius-Young formula asserts
\[
f(\lambda) = n! \prod_{i > j} (h_i - h_j) \prod_{i=1}^r \frac{1}{h_i!}
\]
\begin{remark}
The formula which tends to get taught in combinatorics classes is the \emph{hook length formula}: for each position $x$ in the standard Young tableau, let $hook(x)$ count the number of cells to the right of it, plus the number of cells below it, plus 1 for $x$ itself. Then $f(\lambda) = n! / \prod_x hook(x)$, but this isn't very useful for us. (It is useful in other contexts. For example, Vershik and Kerov used this formula to obtain the asymptotic shape of a typical standard tableau.)
\end{remark}
In particular, by the RSK correspondence the number of permutations $\pi$ with $r(\pi) = r$ will be
\[
(n!)^2 \sum_{h_1,\dotsc,h_r: \sum h_i = n+\frac{r(r-1)}{2}} \prod_{i < j}(h_i - h_j)^2 \prod_{i=1}^r \frac{1}{(h_i!)^2}
\]
Since the product $\prod (h_i - h_j)^2$ is related to the Vandermonde determinant, we get the connection to orthogonal polynomials and then to random matrices.

\subsubsection{Last passage percolation.}
Consider a square $M \times N$ lattice (with $(M+1) \times (N+1)$ points), with weights $w_{ij}$ in vertices. We will take $w_{ij}$ to be iid geometric, i.e.
\[
\BP(w_{ij} = k) = (1-q)q^k.
\]
We would like to find a path from $(0,0)$ to $(M,N)$ which moves only up and to the right, and which maximizes the sum $\sum w_{ij}$ of the weights it passes through. Let
\[
G(N,M) = \max_{p: \text{path up and right}} \sum_{w_{ij}}
\]
be this maximum.

Note that $W$ is a matrix of nonnegative integers. We can write a generalised permutation corresponding to it: the permutation $\pi$ will contain the pair $\binom{i}{j}$ $w_{ij}$ times. For example,
\[
W = \begin{pmatrix} 1 & 2 & 0\\0 & 3 & 0\\ 1 & 1 & 0\\ 1 & 0 & 1\end{pmatrix},
\quad
\pi = \begin{pmatrix}
1 & 1 & 1 & 2 & 2 & 2 & 3 & 3 & 4 & 4\\
1 & 2 & 2 & 2 & 2 & 2 & 1 & 2 & 1 & 3
\end{pmatrix}
\]
Then the optimal path corresponds to the longest nondecreasing subsequence in $\pi$, and $\pi$ (via the RSK correspondence -- in fact, this was Knuth's contribution) corresponds to a pair of \emph{semistandard} Young tableaux: the semistandard Young tableaux are filled by the top, resp. the bottom, row of $\pi$, and the numbers must be increasing along rows and nondecreasing along columns. Semistandart Young tableaus can be counted similarly to standard ones; the polynomials that arise are called the Meixner orthogonal polynomials.

\subsubsection{Other problems.}
We don't have time to mention these in any detail, but there are totally asymmetric exclusion processes, Aztec diamond domino tilings, and viscious walkers. 
Most of the work on the topics in this section has been done by Johansson and his co-authors, and the best place to find out more is to read his papers, for example, \cite{baik_deift_johansson99}, \cite{johansson00}, \cite{johansson01a}, \cite{johansson02}.

\bibliographystyle{plain}
\bibliography{comtest}
\end{document}